\documentclass[final]{jotart}
\usepackage{amsmath}
\usepackage{amssymb}
\usepackage{IEEEtrantools}
\usepackage[arrow, matrix, curve]{xy}
\usepackage[hidelinks]{hyperref}
\theoremstyle{proclaim}
\newtheorem{theorem}{Theorem}[section]
\newtheorem{lemma}[theorem]{Lemma}

\newtheorem{proposition}[theorem]{Proposition}
\theoremstyle{statement}
\newtheorem{remark}[theorem]{Remark}
\newtheorem{definition}[theorem]{Definition}
\newtheorem{example}[theorem]{Example}
\theoremstyle{fancyproclaim}

\numberwithin{equation}{section}

\newcommand{\id}{\operatorname{id}}
\newcommand{\ev}{\operatorname{ev}}
\newcommand{\ran}{\operatorname{ran}}
\newcommand{\tpartial}{\widetilde{\partial}}
\newcommand{\hotimes}{\mathbin{\widehat{\otimes}}}
\newcommand{\la}{\curvearrowright}
\newcommand{\ra}{\curvearrowleft}
\newcommand{\C}{\mathbb{C}}
\newcommand{\D}{\mathcal{D}}
\newcommand{\M}{\mathcal{M}}
\newcommand{\N}{\mathbb{N}}
\renewcommand{\O}{\mathcal{O}}
\begin{document}
\issueinfo{00}{0}{0000} 
\commby{Editor}
\pagespan{101}{110}
\date{Month dd, yyyy}
\title[Free and Cyclic Differential Calculus]{A Note on the Free and Cyclic Differential Calculus}
\dedicatory{Dedicated to Dan-Virgil Voiculescu on the occasion of his 70th birthday.}
\author[T. Mai {\protect \and} R. Speicher]{Tobias Mai {\protect
\and} Roland Speicher}
\address{TOBIAS MAI, Saarland University, Department of Mathematics,\linebreak D-66123 Saarbr\"ucken, Germany}
\email{mai@math.uni-sb.de}
\address{ROLAND SPEICHER, Saarland University, Department of Mathematics, D-66123 Saarbr\"ucken, Germany}
\email{speicher@math.uni-sb.de}
\begin{abstract} 
In 2000, Voiculescu proved an algebraic characterization of cyclic gradients of noncommutative polynomials. We extend this remarkable result in two different directions: first, we obtain an analogous characterization of free gradients; second, we lift both of these results to Voiculescu's fundamental framework of multivariable generalized difference quotient rings.
For that purpose, we develop the concept of divergence operators, for both free and cyclic gradients, and study the associated (weak) grading and cyclic symmetrization operators, respectively.
One the one hand, this puts a new complexion on the initial polynomial case, and on the other hand, it provides a uniform framework within which also other examples -- such as a discrete version of the It\^o stochastic integral -- can be treated.
\end{abstract}
\maketitle

\section{INTRODUCTION}

The analytic treatment of free probability theory rests on the noncommutative differential calculus for the noncommutative and cyclic derivatives.

On the unital complex algebra $\C\langle x_1,\dots,x_n\rangle$ of all noncommutative polynomials in $n$ formal non-commuting variables $x_1,\dots,x_n$, the \emph{noncommutative derivatives} $\partial_1,\dots,\partial_n$ are certain derivations which take their values in the $\C\langle x_1,\dots,x_n\rangle$-bimodule $\C\langle x_1,\dots,x_n\rangle \otimes \C\langle x_1,\dots,x_n\rangle$, while the \emph{cyclic derivatives} $\D_1,\dots,\D_n$ are linear maps which are canonically associated to $\partial_1,\dots,\partial_n$; see Section \ref{subsec:NC_cyclic_derivatives} for the precise definitions.

In classical analysis, the Poincar\'e lemma provides, loosely spoken, a criterion to decide if a given vector-valued function is the gradient of a scalar-valued function. In this paper, we are concerned with noncommutative (algebraic) counterparts of this result, where the gradient is built with respect to the noncommutative and the cyclic derivatives.

For the case of cyclic gradients, Voiculescu obtained such a characterization in \cite{Voi00a}. His proof is based on explicit computations on the level of monomials. Our aim is to give an alternative approach that relies more on the arithmetic and the differential calculus on the algebra of noncommutative polynomials. In fact, we use the fundamental language of \emph{generalized difference quotient rings}, as introduced by Voiculescu in \cite{Voi00b} and refined in \cite{Voi04}, in order to uncover the general mechanism behind this result, hoping to open the door for further interesting applications and generalizations, for instance in the context of \cite{Voi02,Voi04,Voi10,Voi19}. We point out that, especially in the case of a single variable, generalized difference quotient rings are also known as \emph{infinitesimal coalgebras} \cite{JR79}, \emph{Newtonian coalgebras} \cite{ER98}, and \emph{infinitesimal or $\epsilon$-bialgebras} \cite{A00}; see also \cite{L03}, where the multivariable case was addressed, and \cite{Voi10}.

In either case, it becomes relevant that $\C\langle x_1,\dots,x_n\rangle$ is in fact a graded algebra with respect to the natural grading induced by the polynomial degree, which corresponds to the decomposition into the eigenspaces of the \emph{number operator} $N$; see Section \ref{subsec:number_operator} for the precise definitions.
More precisely, we reprove in this way the following theorem, which is essentially \cite[Theorem 1]{Voi00a}.

\begin{theorem}\label{thm:characterization_cyclic_gradients}
Let $p=(p_1,\dots,p_n) \in \C\langle x_1,\dots,x_n\rangle^n$ be an $n$-tuple of noncommutative polynomials. Then the following statements are equivalent:
\begin{enumerate}
 \item\label{it:cyclic_gradients-i} $p$ is a cyclic gradient, i.e., there exists $q\in \C\langle x_1,\dots,x_n\rangle$ such that $$p_i = \D_i q \qquad\text{for all $i=1,\dots,n$}.$$
 \item\label{it:cyclic_gradients-ii} We have that $$\sum^n_{j=1} [x_j, p_j] = 0.$$
 \item\label{it:cyclic_gradients-iii} For $i,j=1,\dots,n$, we have that $$\partial_i p_j = \sigma \big(\partial_j p_i\big).$$
 \item\label{it:cyclic_gradients-iv} For $i=1,\dots,n$, it holds true that $$\D_i \Big( \sum^n_{j=1} x_j p_j \Big) = (N+\id) p_i.$$
\end{enumerate}
When the equivalent conditions are satisfied, we can find a noncommutative polynomial $q\in \C\langle x_1,\dots,x_n\rangle$ satisfying $\D q = p$ by solving the equation $N q = \sum^n_{j=1} x_j p_j$.
\end{theorem}

We emphasize that Theorem \ref{thm:characterization_cyclic_gradients} adds the ``integrability condition'' formulated in \eqref{it:cyclic_gradients-iii} to the list of equivalent conditions characterizing cyclic gradients as given in \cite[Theorem 1]{Voi00a}; it can be seen as a counterpart of the classical Schwarz integrability condition. Further, we note that $\sigma$ stands for the \emph{flip mapping} on $\C\langle x_1,\dots,x_n\rangle \otimes \C\langle x_1,\dots,x_n\rangle$; see Section \ref{subsec:bimodules}.

\begin{remark}
The integrability condition \eqref{it:cyclic_gradients-iii} is of particular interest in view of some fundamental and still open question about conjugate variables which was formulated by Voiculescu in \cite{Voi00a} and which goes back at least to \cite{CDG01}. Let $X_1,\dots,X_n$ be selfadjoint noncommutative random variables in some tracial $W^\ast$-probability space $(\M,\tau)$ with finite non-microstates free Fisher information and let $\xi = (\xi_1,\dots,\xi_n) \in L^2(W^\ast(X),\tau)^n$ be the system of conjugate variables for $X = (X_1,\dots,X_n)$; see \cite{Voi98} for the precise definitions. The question is whether the tuple $\xi$ lies in the closure of the linear space $\{(\D p)(X) \mid p\in \C\langle x_1,\dots,x_n\rangle\}$ with respect to the norm of the Hilbert space $L^2(W^\ast(X),\tau)^n$.

As it was shown in \cite[Corollary 5.12]{Voi99} that $\sum^n_{j=1} [X_j, \xi_j] = 0$, Voiculescu's characterization of cyclic gradients \cite[Theorem 1]{Voi00a} answers this question to the affirmative in the particular case when the conjugate variables belong to the unital algebra generated by $X$.

In order to handle the general case, Voiculescu's result needs to be extended to (a sufficiently large subspace of) $L^2(W^\ast(X),\tau)$. There is some hope that our variable-free approach will be helpful in that respect, because it is based only on identities for noncommutative differential operators which are likely to carry over to correctly chosen closures.
For that purpose, as our Theorem \ref{thm:characterization_cyclic_gradients_GDQ} and its proof show, the integrability condition in Theorem \ref{thm:characterization_cyclic_gradients} \eqref{it:cyclic_gradients-iii} seems to more accessible than the condition formulated in Theorem \ref{thm:characterization_cyclic_gradients} \eqref{it:cyclic_gradients-ii}.
Remarkably, it was shown in \cite[Lemma 36]{Dab14} that $\overline{\partial}_i \xi_j = \sigma( \overline{\partial}_j \xi_i)$ indeed holds for $i,j=1,\dots,n$ whenever the conjugate variables belong to the domain of $\overline{\partial}$, which is the closure of the unbounded linear operator
$$\partial:\ L^2(W^\ast(X)) \supseteq \C\langle X\rangle \to (L^2(W^\ast(X),\tau) \otimes L^2(W^\ast(X),\tau))^n;$$
note that $\sigma$ extends uniquely to an isometric automorphism of the Hilbert space $L^2(W^\ast(X),\tau) \otimes L^2(W^\ast(X),\tau)$.
In view of this fact, it is tempting to guess that $L^2$-approximation by cyclic gradients is possible at least for such ``differentiable'' conjugate variables. We leave this for future investigation.

In any case, in comparison with \cite[Corollary 5.12]{Voi99} and \cite[Lemma 36]{Dab14}, \cite[Theorem 1]{Voi00a} and Theorem \ref{thm:characterization_cyclic_gradients} show that the similarity between systems of conjugate variables and cyclic gradients goes quite deep, which gives some evidence to the conjecture.
\end{remark}

Further, we note that \cite[Theorem 1]{Voi00a} provides another equivalent condition which is not listed above; it demands $\sum^n_{j=1} x_j p_j$ to belong to the range of the so-called \emph{cyclic symmetrization operator} $C: \C\langle x_1,\dots,x_n\rangle \to \C\langle x_1,\dots,x_n\rangle$; see Section \ref{subsec:cyclic_symmetrization_operator}. This property, however, cannot be recovered by our approach yet. Nonetheless, it follows from Lemma \ref{lem:characterization_cyclic_gradients_GDQ} that if $p=(p_1,\dots,p_n)$ satisfies the equivalent conditions of Theorem \ref{thm:characterization_cyclic_gradients}, then every $q\in \C\langle x_1,\dots,x_n\rangle$ solving the equation $C q = \sum^n_{j=1} x_j p_j$ satisfies $\D q = p$.

Our second goal is the following counterpart of Theorem \ref{thm:characterization_cyclic_gradients} which characterizes free gradients among all \emph{noncommutative bi-polynomials}, i.e., elements in the tensor product $\C\langle x_1,\dots,x_n\rangle \otimes \C\langle x_1,\dots,x_n\rangle$. We point out that, while both theorems show a very nice similarity, none of them is an immediate consequence of the respective other.

\begin{theorem}\label{thm:characterization_free_gradients}
Let $u=(u_1,\dots,u_n) \in (\C\langle x_1,\dots,x_n\rangle \otimes \C\langle x_1,\dots,x_n\rangle)^n$ be an $n$-tuple of noncommutative bi-polynomials. Then the following statements are equivalent:
\begin{enumerate}
 \item\label{it:free_gradients-i} $u$ is a free gradient, i.e., there exists $q\in \C\langle x_1,\dots,x_n\rangle$ such that $$u_i = \partial_i q \qquad\text{for $i=1,\dots,n$}.$$
 \item\label{it:free_gradients-ii} There exists $q\in \C\langle x_1,\dots,x_n\rangle$ such that $$\sum^n_{j=1} u_j \sharp [x_j, 1\otimes 1] = [q, 1\otimes 1],$$ where 
$[q, 1 \otimes 1] = q \otimes 1 - 1 \otimes q$ and $$(p_1 \otimes p_2) \sharp [x_j, 1 \otimes 1] = p_1 x_j \otimes p_2 - p_1 \otimes x_j p_2.$$ 
 \item\label{it:free_gradients-iii} For $i,j=1,\dots,n$, we have that $$(\id \otimes \partial_i)(u_j) = (\partial_j \otimes \id)(u_i).$$
 \item\label{it:free_gradients-iv} For $i=1,\dots,n$, it holds true that $$\partial_i \Big( \sum^n_{j=1} u_j \sharp x_j \Big) = (N \otimes \id + \id \otimes N + \id \otimes \id)(u_i).$$
\end{enumerate}
In the case that the equivalent conditions are satisfied, the noncommutative polynomials $q$ in \eqref{it:free_gradients-i} and \eqref{it:free_gradients-ii} differ only by an additive constant, so that every $q$ which satisfies one of these conditions satisfies also the other; such a $q\in\C\langle x_1,\dots,x_n\rangle$ can be found by solving the equation $N q = \sum^n_{j=1} u_j \sharp x_j$.
\end{theorem}

We note that both results remain true when the noncommutative polynomials $\C\langle x_1,\dots,x_n\rangle$ are replaced by the algebra $\C\langle x_1,\dots,x_n; R\rangle$ of all formal power series with prescribed radius of convergence $R>0$. The algebraic tensor product $\otimes$ must then be replaced by the projective tensor product $\hotimes$.

\begin{example}
We consider the noncommutative polynomial $q \in \C\langle x_1,x_2 \rangle$ which is given by $q = x_1 x_2^2 x_1$. We have then
\begin{align*}
u_1 := \partial_1 q &= 1 \otimes x_2^2 x_1 + x_1 x_2^2 \otimes 1,\\
u_2 := \partial_2 q &= x_1 \otimes x_2 x_1 + x_1 x_2 \otimes x_1,
\end{align*}
and hence
\begin{align*}
p_1 := \D_1 q &= x_2^2 x_1 + x_1 x_2^2,\\
p_2 := \D_2 q &= x_2 x_1^2 + x_1^2 x_2.
\end{align*}

1) We want to use Theorem \ref{thm:characterization_cyclic_gradients} to ``recover'' $q$ from the tuple $p=(p_1,p_2)$. In order to do so, we first verify that $p$ satisfies the integrability condition formulated in Item \eqref{it:cyclic_gradients-iii}; since
\begin{align*}
\partial_1 p_1 &= x_2^2 \otimes 1 + 1 \otimes x_2^2,\\
\partial_2 p_1 &= 1 \otimes x_2 x_1 + x_2 \otimes x_1 + x_1 \otimes x_2 + x_1 x_2 \otimes 1,\\
\partial_1 p_2 &= x_2 \otimes x_1 + x_2 x_1 \otimes 1 + 1 \otimes x_1 x_2 + x_1 \otimes x_2,\\
\partial_2 p_2 &= 1 \otimes x_1^2 + x_1^2 \otimes 1,
\end{align*}
we see that indeed $\partial_1 p_2 = \sigma(\partial_2 p_1)$, while $\partial_1 p_1$ and $\partial_2 p_2$ are invariant under $\sigma$. Next, we compute
$$x_1 p_1 + x_2 p_2 = x_1 x_2^2 x_1 + x_1^2 x_2^2 + x_2^2 x_1^2 + x_2 x_1^2 x_2$$
and we chose a noncommutative polynomial $\tilde{q}$ such that $N \tilde{q} = x_1 p_1 + x_2 p_2$, say
$$\tilde{q} = \frac{1}{4}\big(x_1 x_2^2 x_1 + x_1^2 x_2^2 + x_2^2 x_1^2 + x_2 x_1^2 x_2\big).$$
Then $\D \tilde{q} = p$; we note that $\tilde{q}$ does not agree with $q$ but is rather a cyclically symmetrized version of $q$. Thus, it is more appropriate to solve the equation $x_1 p_1 + x_2 p_2 = C \tilde{q}$, for which we easily find the solution $\tilde{q} = q$. 

2) Analogously, we can apply Theorem \ref{thm:characterization_free_gradients} to reconstruct $q$ from the tuple $u=(u_1,u_2)$. First, we check that $u$ satisfies the condition formulated in Item \eqref{it:free_gradients-iii}; indeed,
\begin{IEEEeqnarray*}{lCCCr}
(\partial_1 \otimes \id)(u_1) &= & 1 \otimes x_2^2 \otimes 1 &= &(\id \otimes \partial_1)(u_1),\\
(\partial_2 \otimes \id)(u_1) &= & x_1 \otimes x_2 \otimes 1 + x_1 x_2 \otimes 1 \otimes 1 &= &(\id \otimes \partial_1)(u_2),\\
(\partial_1 \otimes \id)(u_2) &= & 1 \otimes 1 \otimes x_2 x_1 + 1 \otimes x_2 \otimes x_1 &= &(\id \otimes \partial_2)(u_1),\\
(\partial_2 \otimes \id)(u_2) &= & x_1 \otimes 1 \otimes x_1 &= &(\id \otimes \partial_2)(u_2).
\end{IEEEeqnarray*}
One easily confirms that $u_1 \sharp x_1 + u_2 \sharp x_2 = 4 q = N q$; we conclude, as expected, that $\partial q = u$.
\end{example}

The rest of this paper is organized as follows.

In Section \ref{sec:preliminaries}, we collect some preliminaries on bimodules, derivations, and tensor products.
Sections \ref{sec:DiffCalculus_GDQ} and \ref{sec:gradients} constitute the main part of this paper. In Section \ref{sec:DiffCalculus_GDQ}, we provide some background on Voiculescu's multivariable generalized difference quotient rings and we develop the theory of divergence operators, both for free and cyclic gradients. Building on the latter, we obtain in Section \ref{sec:gradients} the announced generalizations of Theorems \ref{thm:characterization_cyclic_gradients} and \ref{thm:characterization_free_gradients}.
In Section \ref{sec:gradients_NCPolys}, we apply these general results in the particular setting of noncommutative polynomials; in this way, we finally prove Theorems \ref{thm:characterization_cyclic_gradients} and \ref{thm:characterization_free_gradients}. In preparation of this, we explain in Section \ref{sec:NCPolys} how noncommutative polynomials fit into the general framework developed in Section \ref{sec:DiffCalculus_GDQ}.
In Section \ref{sec:embeddings}, we discuss the phenomenon that a divergence on a multivariable generalized difference quotient ring induces an embedding of the ring of noncommutative polynomials; in particular, we provide some further examples for situations in which our approach can be applied.
The last section, Section \ref{sec:vanishing_gradients}, is devoted to the study of kernels of free and cyclic gradients.

\section{PRELIMINARIES}\label{sec:preliminaries}

\subsection{Derivations}\label{subsec:derivations}

Let $A$ be a complex algebra and let $M$ be an $A$-bimodule. A \emph{derivation} $d: A \to M$ means a linear map which further satisfies the \emph{Leibniz rule} $d(a_1 a_2) = d(a_1) \cdot a_2 + a_1 \cdot d(a_2)$ for all $a_1,a_2\in A$, where $\cdot$ stands for the bimodule action of $A$ on $M$.

Let $\mu: A \otimes A \to A$ be the \emph{multiplication map} which is the linear map defined by linear extension of $\mu(a_1 \otimes a_2) := a_1 a_2$. Further, we let 
$\mu_\la: A \otimes M \to M$ and $\mu_\ra: M \otimes A \to A$ implement the left and right actions, respectively, of $A$ on $M$, i.e., $\mu_\la(a \otimes m) := a \cdot m$ and $\mu_\ra(m \otimes a) := m \cdot a$. Then, the Leibniz rule can be rephrased as an identity on $A \otimes A$ by
\begin{equation}\label{eq:derivation}
d \circ \mu = \mu_\ra \circ (d \otimes \id_A) + \mu_\la \circ (\id_A \otimes d).
\end{equation}

\subsection{Tensor products and bimodules}\label{subsec:bimodules}

Let $A$ be a complex algebra and let $M$ be an $A$-bimodule. We introduce $\sharp: (A \otimes A) \times M \to M$ as a bilinear map which is determined by $(a_1 \otimes a_2) \sharp m := a_1 \cdot m \cdot a_2$ and we define for $a\in A$ and $m\in M$ the commutator $[a,m] := a \cdot m - m \cdot a$.

A particularly important instance is $A\otimes A$, namely the algebraic tensor product over $\C$ of $A$ with itself, which forms an $A$-bimodule with the canonical left and right action $\cdot$ which is determined by $b_1 \cdot (a_1 \otimes a_2) \cdot b_2 = (b_1 a_1) \otimes (a_2 b_2)$; note that accordingly $\mu_\la = \mu \otimes \id_A$ and $\mu_\ra = \id_A \otimes \mu$. In this case, we have that $(b_1 \otimes b_2) \sharp (a_1 \otimes a_2) = (b_1 a_1) \otimes (a_2 b_2)$ and $[b,a_1\otimes a_2] = (b a_1) \otimes a_2 - a_1 \otimes (a_2 b)$.
We note that correspondingly $(A \otimes A)^n$ becomes an $A$-bimodule with respect to the entry-wise action defined by $a_1 \cdot (u_1,\dots,u_n) \cdot a_2 := (a_1 \cdot u_1 \cdot a_2, \dots, a_1 \cdot u_n \cdot a_2)$.

For a general $A$-bimodule $M$, we further define
$$\sharp_1:\ (A \otimes A \otimes A) \times M \to M \otimes A,\quad (a_1 \otimes a_2 \otimes a_3) \sharp_1 m := (a_1 \cdot m \cdot a_2) \otimes a_3$$
and
$$\sharp_2:\ (A \otimes A \otimes A) \times M \to A \otimes M,\quad (a_1 \otimes a_2 \otimes a_3) \sharp_2 m := a_1 \otimes (a_2 \cdot m \cdot a_3).$$

Moreover, for any $n\in\N$, we consider the canonical left action $\pi: S_n \times A^{\otimes n} \to A^{\otimes n}$ of the symmetric group $S_n$ on the $n$-fold tensor product $A^{\otimes n}$ of $A$ with itself. More precisely, to every permutation $\sigma \in S_n$ we associate the linear map $\pi_\sigma: A^{\otimes n} \to A^{\otimes n}$ which is induced by $\pi_\sigma(a_1 \otimes \dots \otimes a_n) = a_{\sigma^{-1}(1)} \otimes \dots \otimes a_{\sigma^{-1}(n)}$. Note that indeed $\pi_{\sigma_1 \circ \sigma_2} = \pi_{\sigma_1} \circ \pi_{\sigma_2}$ for every $\sigma_1,\sigma_2 \in S_n$.
Further, we point out that $\pi_{(12)}: A \otimes A \to A \otimes A$ is nothing but the well-known \emph{flip mapping} on $A\otimes A$, i.e., $\pi_{(12)}(a_1 \otimes a_2) = a_2 \otimes a_1$; following the usual convention, we will denote $\pi_{(12)}$ by $\sigma$.

\section{DIFFERENTIAL CALCULUS IN MULTIVARIABLE GDQ RINGS}\label{sec:DiffCalculus_GDQ}

In his fundamental paper \cite{Voi00b} (see also \cite{Voi04}), Voiculescu introduced the so-called generalized difference quotient rings as a general frame to deal with algebras that allow a differential calculus analogous to the noncommutative derivatives on the algebra of noncommutative polynomials. We will use this language to uncover the mechanism behind Theorems \ref{thm:characterization_cyclic_gradients} and \ref{thm:characterization_free_gradients} and to provide proofs in a variable-free manner.

\subsection{Voiculescu's multivariable GDQ rings}

In this subsection, we recall for the reader's convenience the basic definitions of \cite{Voi00b} in the refined version of \cite{Voi04}. Note that we start right from the beginning with the multivariable case.

\begin{definition}\label{def:GDQring}
A \emph{multivariable generalized difference quotient ring} (a \emph{multivariable GDQ ring}, for short) is a triple $(A,\mu,\partial)$ consisting of a complex algebra $A$, the induced multiplication mapping $\mu: A\otimes A \to A$, and an $n$-tuple $\partial=(\partial_1,\dots,\partial_n)$ of linear mappings
$$\partial_i: A \to A \otimes A, \qquad i=1,\dots,n,$$
that enjoy the following properties:
\begin{enumerate}
 \item The mappings $\partial_1,\dots,\partial_n$ satisfy the \emph{joint coassociativity relation}
 \begin{equation}\label{eq:GDQring_coass}
 (\partial_i \otimes \id_A) \circ \partial_j = (\id_A \otimes \partial_j) \circ \partial_i \qquad\text{for $i,j=1,\dots,n$}.
 \end{equation}
 \item\label{it:GDQring_deriv} Each of the mappings $\partial_1,\dots,\partial_n$ is a derivation on $(A,\mu)$, i.e., we have that
 \begin{equation}\label{eq:GDQring_deriv}
 \partial_i \circ \mu = (\mu \otimes \id_A)\circ(\id_A \otimes \partial_i) + (\id_A \otimes \mu) \circ (\partial_i \otimes \id_A) \qquad\text{for $i=1,\dots,n$}.
 \end{equation}
\end{enumerate}
We refer to $\partial$ as the \emph{gradient of $A$}, which we may view as a linear map $\partial: A \to (A \otimes A)^n$.
If $A$ has a unit element $1_A$, we call the multivariable GDQ ring $(A,\mu,\partial)$ \emph{unital}.
\end{definition}

We point out that Item \eqref{it:GDQring_deriv} just rephrases the \emph{Leibniz rule}
\begin{equation}\label{eq:GDQring_Leibniz}
\partial_i(a_1 a_2) = (\partial_i a_1) \cdot a_2 + a_1 \cdot (\partial_i a_2) \qquad\text{for all $a_1,a_2\in A$}
\end{equation}
in the sense of \eqref{eq:derivation} with respect to the canonical $A$-bimodule structure on $A \otimes A$; see Section \ref{subsec:bimodules}. In fact, if $(A\otimes A)^n$ is viewed as an $A$-bimodule, then also the gradient $\partial: A \to (A \otimes A)^n$ becomes a derivation.

The Leibniz rule \eqref{eq:GDQring_Leibniz} enforces that in the unital case $\partial_i 1_A = 0$ for $i=1,\dots,n$. For later use, we point out that \eqref{eq:GDQring_Leibniz} further implies that, for each $a\in A$ and each $u\in A \otimes A$,
\begin{equation}\label{eq:GDQring_Leibniz-sharp}
\partial_i (u\sharp a) = (\id \otimes \partial_i)(u) \sharp_1 a + (\partial_i \otimes \id)(u) \sharp_2 a + u \sharp (\partial_i a).
\end{equation}

Another notion from \cite{Voi00b,Voi04} which will play an important role in our considerations is that of a grading on a multivariable GDQ ring.

\begin{definition}\label{def:gradedGDQ}
A GDQ ring $(A,\mu,\partial)$ is called \emph{graded}, if there exits a linear mapping $L: A \to A$, the so-called \emph{grading}, that satisfies the following properties:
\begin{enumerate}
 \item\label{it:numop_deriv} The associated \emph{number operator} $N := L-\id_A$ is an $A$-valued derivation on $A$ with respect to $\mu$, i.e.,
 \begin{equation}\label{eq:numop_deriv}
 N \circ \mu = \mu \circ (N \otimes \id_A + \id_A \otimes N).
 \end{equation}
 \item\label{it:numop_coderiv} $L$ is a coderivation with respect to each $\partial_i$, i.e., we have that
 \begin{equation}\label{eq:numop_coderiv-L}
 \partial_i \circ L = (L \otimes \id_A + \id_A \otimes L) \circ \partial_i \qquad\text{for all $i=1,\dots,n$},
 \end{equation}
 or equivalently, in terms of the number operator $N$,
 \begin{equation}\label{eq:numop_coderiv-N}
 \partial_i \circ N = (N \otimes \id_A + \id_A \otimes N + \id_{A \otimes A}) \circ \partial_i \qquad\text{for all $i=1,\dots,n$}.
 \end{equation}
\end{enumerate}
We say that $(A,\mu,\partial)$ is \emph{weakly graded}, if a linear mapping $L: A \to A$ exists which satisfies \eqref{it:numop_coderiv} but not necessarily \eqref{it:numop_deriv}; we call $L$ then a \emph{weak grading} of $(A,\mu,\partial)$.
\end{definition}

\subsection{The cyclic derivatives on multivariable GDQ rings}

The aim of this subsection is to transfer the notion of cyclic derivatives to the framework of multivariable GDQ rings.

\begin{definition}\label{def:GDQring_cyclic_derivatives}
Let $(A,\mu,\partial)$ be a multivariable GDQ ring. The linear maps $\D_j: A \to A$ which are defined by $\D_i := \mu \circ \sigma \circ \partial_i$ for $i=1,\dots,n$ are called the \emph{cyclic derivatives associated to $\partial$}. We call $\D = (\D_1, \dots, \D_n)$ the \emph{cyclic gradient}, which may be viewed as a linear map $\D: A \to A^n$.
\end{definition}

For a multivariable GDQ ring $(A,\mu,\partial)$, one can also study the maps $\mu \circ \partial_i: A \to A$ for $i=1,\dots,n$, which differ from the cyclic derivatives $\D_i = \mu \circ \sigma \circ \partial_i$ only by the flip map $\sigma$; we do not go into details here but we refer the interested reader to \cite{A00}.

While each $\partial_i$ for a multivariable GDQ ring $(A,\mu,\partial)$ satisfies by Definition \ref{def:GDQring} the Leibniz rule in the form of \eqref{eq:GDQring_Leibniz}, the corresponding rule for the cyclic derivatives is slightly more delicate. In fact, we have that
\begin{equation}\label{eq:GDQring_Leibniz-cyclic}
\D_i(a_1 a_2) = (\tpartial_i a_1) \sharp a_2 + (\tpartial_i a_2) \sharp a_1 \qquad\text{for all $a_1,a_2\in A$},
\end{equation}
where we put $\tpartial_i := \sigma \circ \partial_i$ for $i=1,\dots,n$; note that clearly $\D_i = \mu \circ \tpartial_i$. Indeed, \eqref{eq:GDQring_Leibniz-cyclic} follows from \eqref{eq:GDQring_Leibniz} by using that $(\mu \circ \sigma)(u \cdot a) = \sigma(u) \sharp a$ and $(\mu \circ \sigma)(a \cdot u) = \sigma(u) \sharp a$ for all $u\in A \otimes A$ and $a\in A$.

Cyclic derivatives inherit also the joint coassociativity relation \eqref{eq:GDQring_coass} satisfied by the derivatives on the underlying multivariable GDQ ring; this is the content of the next lemma.

\begin{lemma}\label{lem:GDQring_coass_cyclic}
Let $(A,\mu,\partial)$ be a multivariable GDQ ring and consider the cyclic derivatives $\D=(\D_1,\dots,\D_n)$ associated to $\partial$. Then
$$\sigma \circ \partial_i \circ \D_j = \partial_j \circ \D_i \qquad\text{for $i,j=1,\dots,n$}.$$
\end{lemma}

\begin{proof}
To begin with, we note that for $i=1,\dots,n$
$$(\partial_i \otimes \id_A) \circ \sigma = \pi_{(132)} \circ (\id_A \otimes \partial_i) \qquad\text{and}\qquad (\id \otimes \partial_i) \circ \sigma = \pi_{(123)} \circ (\partial_i \otimes \id_A).$$
Furthermore, one easily sees that
$$\sigma \circ (\id_A \otimes \mu) = (\mu \otimes \id_A) \circ \pi_{(132)} \qquad\text{and}\qquad \sigma \circ (\mu \otimes \id_A) = (\id_A \otimes \mu) \circ \pi_{(123)}.$$
Using these relations as well as the joint coassociativity relation \eqref{eq:GDQring_coass}, we may check for $i,j=1,\dots,n$ that
\begin{align*}
\sigma \circ (\mu \otimes \id_A) \circ (\id_A \otimes \partial_i) \circ \sigma \circ \partial_j
&= (\id_A \otimes \mu) \circ \pi_{(123)} \circ (\id_A \otimes \partial_i) \circ \sigma \circ \partial_j\\
&= (\id_A \otimes \mu) \circ \pi_{(123)} \circ \pi_{(123)} \circ (\partial_i \otimes \id_A) \circ \partial_j\\
&= (\id_A \otimes \mu) \circ \pi_{(132)} \circ (\id_A \otimes \partial_j) \circ \partial_i\\
&= (\id_A \otimes \mu) \circ (\partial_j \otimes \id_A) \circ \sigma \circ \partial_i
\end{align*}
and analogously
$$\sigma \circ (\id_A \otimes \mu) \circ (\partial_i \otimes \id_A) \circ \sigma \circ \partial_j = (\mu \otimes \id_A) \circ (\id \otimes \partial_j) \circ \sigma \circ \partial_i.$$
Combining the latter identities with \eqref{eq:GDQring_deriv} yields
\begin{align*}
\sigma \circ \partial_i \circ \D_j &= \sigma \circ (\partial_i \circ \mu) \circ \sigma \circ \partial_j\\
                                   &= \sigma \circ \big( (\mu \otimes \id_A) \circ (\id_A \otimes \partial_i) + (\id_A \otimes \mu) \circ (\partial_i \otimes \id_A)\big) \circ \sigma \circ \partial_j\\
                                   &= \big((\id_A \otimes \mu) \circ (\partial_j \otimes \id_A) + (\mu \otimes \id_A) \circ (\id \otimes \partial_j) \big) \circ \sigma \circ \partial_i\\
                                   &= (\partial_j \circ \mu) \circ \sigma \circ \partial_i\\
                                   &= \partial_j \circ \D_i,        
\end{align*}
which is the asserted identity.
\end{proof}

Next, we address the question how the cyclic derivatives interact with the grading in the case of a graded multivariable GDQ ring. The following lemma shows that an easy commutation relation holds.

\begin{lemma}\label{lem:number_operator_cyclic_derivative}
Let $(A,\mu,\partial)$ be a graded multivariable GDQ ring with the grading $L$ and the associated number operator $N$. Then, the cyclic derivatives $\D_1,\dots,\D_n$ associated to $\partial$ satisfy the commutation relation
\begin{equation}\label{eq:number_operator_cyclic_derivative}
\D_i \circ N = L \circ \D_i \qquad\text{for $i=1,\dots,n$.}
\end{equation}
\end{lemma}

\begin{proof}
Because $N$ is a derivation as guaranteed by Definition \ref{def:gradedGDQ} \eqref{it:numop_deriv}, the asserted commutation relation \eqref{eq:number_operator_cyclic_derivative} follows from the fact that $L$ is a coderivation by Definition \ref{def:gradedGDQ} \eqref{it:numop_coderiv}; indeed, using \eqref{eq:numop_deriv}, we get for $i=1,\dots,n$ that
\begin{align*}
\D_i \circ N &= \mu \circ \sigma \circ \partial_i \circ N\\
             &= \mu \circ \sigma \circ \big(N \otimes \id_A + \id_A \otimes N + \id_{A\otimes A}\big) \circ \partial_i\\
             &= \mu \circ \big(N \otimes \id_A + \id_A \otimes N + \id_{A \otimes A}\big) \circ \sigma \circ \partial_i\\
             &= (N + \id_A)\circ \mu \circ \sigma \circ \partial_i\\
             &= L \circ \D_i,
\end{align*}
which is what we wanted to show.
\end{proof}

\subsection{The divergence on multivariable GDQ rings}

At the core of Theorems \ref{thm:characterization_cyclic_gradients} and \ref{thm:characterization_free_gradients} is the formulation of a ``universal rule'' which allows one to find an antiderivative for the presumptive (cyclic) gradient. This is achieved by divergence operators, which we introduce next.

\begin{definition}\label{def:GDQring_divergence}
Let $(A,\mu,\partial)$ be a multivariable GDQ ring. An $n$-tuple $\partial^\star = (\partial_1^\star, \dots, \partial_n^\star)$ of linear mappings
$$\partial_i^\star:\ A \otimes A \to A, \qquad i=1,\dots,n$$
is called a \emph{divergence for $(A,\mu,\partial)$} if the identity
\begin{equation}\label{eq:GDQring_divergence}
\partial_j \circ \partial_i^\star = (\partial_i^\star \otimes \id_A) \circ (\id_A \otimes \partial_j) + (\id_A \otimes \partial_i^\star) \circ (\partial_j \otimes \id_A) + \delta_{i,j} \id_{A \otimes A}
\end{equation}
holds true for every choice of indices $i,j=1,\dots,n$. We will view the divergence as a linear map $\partial^\star: (A \otimes A)^n \to A$ by
$$\partial^\star(u) := \sum^n_{j=1} \partial_j^\star(u_j) \qquad\text{for each $u=(u_1,\dots,u_n) \in (A\otimes A)^n$}.$$
\end{definition}

\begin{remark}\label{rem:prototypical_divergence}
Let $(A,\mu,\partial)$ be a multivariable GDQ ring which is unital with unit element $1_A$. If $\partial^\star = (\partial_1^\star, \dots, \partial_n^\star)$ is a divergence for $(A,\mu,\partial)$, then $a_i := \partial_i^\star(1_A \otimes 1_A)$ for $i=1,\dots,n$ yields elements $a_1,\dots,a_n\in A$ which, according to \eqref{eq:GDQring_divergence}, satisfy
\begin{equation}\label{eq:free_variables}
\partial_j a_i = \delta_{i,j} 1_A \otimes 1_A \qquad\text{for $i,j=1,\dots,n$}.
\end{equation}
Conversely, whenever we find distinguished elements $a_1,\dots,a_n \in A$ with the property \eqref{eq:free_variables}, then a divergence for $(A,\mu,\partial)$ can be defined by $\partial_i^\star(u) := u \sharp a_i$ for $i=1,\dots,n$; this is a prototypical example of a divergence for which the Leibniz rule \eqref{eq:GDQring_Leibniz-sharp} verifies the condition \eqref{eq:GDQring_divergence}.
\end{remark}

Imposing the existence of a divergence turns out to be a rather strong constraint. For instance, it induces automatically a weak grading on the underlying multivariable GDQ ring; the precise statement is given in the following lemma.

\begin{lemma}\label{lem:divergence_induces_grading}
Let $(A,\mu,\partial)$ be a multivariable GDQ ring for which a divergence $\partial^\star$ exists. Define $N: A \to A$ by
$$N := \partial^\star \circ \partial = \sum^n_{i=1} \partial_i^\star \circ \partial_i$$
and $L := N + \id_A$; then $L$ is a weak grading on $(A,\mu,\partial)$. If we suppose in addition that each $\partial_j^\star$ is an $A$-bimodule homomorphism, then $L$ is even a grading on $(A,\mu,\partial)$.
\end{lemma}

\begin{proof}
Let $a\in A$ be given. For $j=1,\dots,n$, we compute using the joint coassociativity relation \eqref{eq:GDQring_coass} of $\partial$ and the defining property of $\partial^\star$ that
\begin{align*}
\partial_j \circ N &= \sum^n_{i=1} (\partial_j \circ \partial_i^\star) \circ \partial_i\\
                   &= \partial_j + \sum^n_{i=1} \big((\partial_i^\star \otimes \id_A) \circ (\id_A \otimes \partial_j) \circ \partial_i + (\id_A \otimes \partial_i^\star) \circ (\partial_j \otimes \id_A) \circ \partial_i\big)\\
                   &= \partial_j + \sum^n_{i=1} \big((\partial_i^\star \otimes \id_A) \circ (\partial_i \otimes \id_A) \circ \partial_j + (\id_A \otimes \partial_i^\star) \circ (\id_A \otimes \partial_i) \circ \partial_j\big)\\
                &= (N \otimes \id_A + \id_A \otimes N + \id_{A \otimes A}) \circ \partial_j,
\end{align*}
from which we conclude that $L$ is a coderivation with respect to $\partial_i$; this verifies that $L$ is a weak grading on $(A,\mu,\partial)$.

If we suppose now in addition that each $\partial_i^\star$ is an $A$-bimodule map, i.e.,
$$\partial_i^\star \circ (\mu \otimes \id_A) = \mu \circ (\id_A \otimes \partial_i^\star) \qquad\text{and}\qquad \partial_i^\star \circ (\id_A \otimes \mu) = \mu \circ (\partial_i^\star \otimes \id_A),$$
then we may check by using \eqref{eq:GDQring_deriv} that
\begin{align*}
N \circ \mu &= \sum^n_{i=1} \partial_i^\star \circ (\partial_i \circ \mu)\\
            &= \sum^n_{i=1} \partial_i^\star \circ \big((\mu \otimes \id_A) \circ (\id_A \otimes \partial_i) + (\id_A \otimes \mu) \circ (\partial_i \otimes \id_A)\big)\\
            &= \mu \circ \sum^n_{i=1} \big((\id_A \otimes \partial_i^\star) \circ (\id_A \otimes \partial_i) + (\partial_i^\star \otimes \id_A) \circ (\partial_i \otimes \id_A)\big)\\
            &= \mu \circ (\id_A \otimes N + N \otimes \id_A),           
\end{align*}
as asserted.
\end{proof}

We point out that besides the canonical divergence which was presented in Remark \ref{rem:prototypical_divergence}, there are other, more ``exotic'' constructions for a divergence. This is the content of the following lemma, which is inspired by Proposition 4.3 in \cite{Voi98}.

\begin{lemma}
Let $(A,\mu,\partial)$ be a unital multivariable GDQ ring and fix any linear functional $\phi: A \to \C$. Suppose that we find elements $a_1,\dots,a_n\in A$ which satisfy $\partial_j a_i = \delta_{i,j} 1_A \otimes 1_A$ for $i,j=1,\dots,n$. For $i=1,\dots,n$, we define $\partial^\star_i: A \otimes A \to A$ by
$$\partial^\star_i(u) := u \sharp a_i - \mu\big((\id_A \otimes \phi \otimes \id_A)\big((\partial_i \otimes \id_A + \id_A \otimes \partial_i)(u)\big)\big)$$
for every $u\in A \otimes A$. Then $\partial^\star=(\partial^\star_1,\dots,\partial^\star_n)$ is a divergence for $(A,\mu,\partial)$.
\end{lemma}

\begin{proof}
For $i=1,\dots,n$, we define linear maps $T_i: A \otimes A \to A$ by
\begin{equation}\label{eq:T-relation}
T_i := \mu \circ (\id_A \otimes \phi \otimes \id_A) \circ (\partial_i \otimes \id_A + \id_A \otimes \partial_i).
\end{equation}
In order to verify the identity \eqref{eq:GDQring_divergence} for $\partial_i^\star$ as defined in the lemma, it suffices to prove that
$$\partial_j \circ T_i = (T_i \otimes \id_A) \circ (\id_A \otimes \partial_j) + (\id_A \otimes T_i) \circ (\partial_j \otimes \id_A)$$
for $i,j=1,\dots,n$. Using the fact that $\partial_j$ is a derivation, we first get that
\begin{align*}
\partial_j \circ T_i
&= (\mu \otimes \id_A) \circ (\id_A \otimes \partial_j) \circ (\id_A \otimes \phi \otimes \id_A) \circ (\partial_i \otimes \id_A + \id_A \otimes \partial_i)\\
& \quad + (\id_A \otimes \mu) \circ (\partial_j \otimes \id_A) \circ (\id_A \otimes \phi \otimes \id_A) \circ (\partial_i \otimes \id_A + \id_A \otimes \partial_i).
\end{align*}
Next, using the joint coassociativity relation, we compute that
\begin{align*}
\lefteqn{(\id_A \otimes \partial_j) \circ (\id_A \otimes \phi \otimes \id_A) \circ (\partial_i \otimes \id_A + \id_A \otimes \partial_i)}\\
&= (\id_A \otimes \phi \otimes \id_A \otimes \id_A) \circ (\id_A \otimes \id_A \otimes \partial_j) \circ  (\partial_i \otimes \id_A + \id_A \otimes \partial_i)\\
&= (\id_A \otimes \phi \otimes \id_A \otimes \id_A) \circ (\partial_i \otimes \id_A \otimes \id_A + \id_A \otimes \partial_i \otimes \id_A) \circ (\id_A \otimes \partial_j)
\end{align*}
and in turn
\begin{align*}
(\mu \otimes \id_A) \circ (\id_A \otimes \partial_j) \circ (\id_A \otimes \phi \otimes \id_A) &\circ (\partial_i \otimes \id_A + \id_A \otimes \partial_i)\\
&= (T_i \otimes \id_A) \circ (\id_A \otimes \partial_j).
\end{align*}
Analogously, we get that
\begin{align*}
(\id_A \otimes \mu) \circ (\partial_j \otimes \id_A) \circ (\id_A \otimes \phi \otimes \id_A) &\circ (\partial_i \otimes \id_A + \id_A \otimes \partial_i)\\
&= (\id_A \otimes T_i) \circ (\partial_j \otimes \id_A).
\end{align*}
In summary, we obtain \eqref{eq:T-relation}, as desired.
\end{proof}

\begin{remark}
In the situation of the previous lemma, suppose that $(A,\phi)$ is a $\ast$-probability space with $\phi$ being tracial and faithful. Further, let $(a_1,\dots,a_n)$ be semicircular system, i.e., a family of freely independent semicircular elements with mean $0$ and variance $1$. Then \cite[Proposition 4.3]{Voi98} tells us that the divergence $\partial^\star$ satisfies $\langle \partial_i a, u \rangle_\phi = \langle a, \partial^\star_i u \rangle_\phi$ for all $a\in A$, $u\in A \otimes A$, and $i=1,\dots,n$, with respect to the induced inner product $\langle a_1, a_2 \rangle_\phi = \phi(a_2^\ast a_1)$. This justifies our $\star$-notation for a divergence.
\end{remark}

Further, we need the notion of a cyclic divergence; the precise definition reads as follows.

\begin{definition}\label{def:GDQring_cyclic_divergence}
Let $(A,\mu,\partial)$ be a multivariable GDQ ring and let $\partial^\star$ be a divergence for $(A,\mu,\partial)$. An $n$-tuple $\D^\star=(\D_1^\star,\dots,\D_n^\star)$ of linear mappings
$$\D_i^\star:\ A \to A, \qquad i=1,\dots,n$$
is called a \emph{cyclic divergence for $(A,\mu,\partial)$ (compatible with $\partial^\star$)} if the identity
$$\D_j \circ \D_i^\star = \partial_i^\star \circ \sigma \circ \partial_j + \delta_{i,j} \id_A$$
holds true for every choice of indices $i,j=1,\dots,n$. We will view the cyclic divergence as a linear map $\D^\star: A^n \to A$ by
$$\D^\star(a) := \sum^n_{i=1} \D_i^\star(a_i) \qquad\text{for each $a=(a_1,\dots,a_n) \in A^n$}.$$
For a given cyclic divergence $\D^\star$, the associated linear map $C: A \to A$ which is defined by
$$C := \D^\star \circ \D = \sum^n_{i=1} \D_i^\star \circ \D_i$$
is called the \emph{cyclic symmetrization operator}.
\end{definition}

Similar to the commutation relation \eqref{eq:number_operator_cyclic_derivative} formulated in Lemma \ref{lem:number_operator_cyclic_derivative}, we have the following result.

\begin{lemma}\label{lem:symmetrization_operator_cyclic_derivative}
Let $(A,\mu,\partial)$ be a multivariable GDQ ring with divergence $\partial^\star$ which is endowed with the induced weak grading $L = \partial^\star \circ \partial + \id_A$ that was constructed in Lemma \ref{lem:divergence_induces_grading}. Consider further a cyclic divergence $\D^\star = (\D_1^\star,\dots,\D_n^\star)$ compatible with $\partial^\star$ and let $C$ be the associated cyclic symmetrization operator. Then, for $j=1,\dots,n$, we have the commutation relation
\begin{equation}\label{eq:symmetrization_operator_cyclic_derivative}
\D_j \circ C = L \circ \D_j.
\end{equation}
\end{lemma}

\begin{proof}
By using defining property of $\D^\star$ and Lemma \ref{lem:GDQring_coass_cyclic}, we get that
\begin{align*}
\D_j \circ C &= \sum^n_{i=1} (\D_j \circ \D_i^\star) \circ \D_i\\
             &= \D_j + \sum^n_{i=1} \partial_i^\star \circ (\sigma \circ \partial_j \circ \D_i)\\
						 &= \D_j + \sum^n_{i=1} (\partial_i^\star \circ \partial_i) \circ \D_j\\
						 &= L \circ \D_j,
\end{align*}
which is the asserted commutation relation.
\end{proof}

In fact, a cyclic divergence can be constructed from a divergence $\partial^\star$ whenever the underlying multivariable GDQ ring $(A,\mu,\partial)$ is unital and the divergence $\partial^\star$ consists of $A$-bimodule homomorphisms; this is explained in the following lemma.
 
\begin{lemma}\label{lem:compatible_cyclic_divergence}
Let $(A,\mu,\partial)$ be a unital multivariable GDQ ring with unit $1_A$ and let $\partial^\star=(\partial_1^\star,\dots,\partial_n^\star$) be a divergence for $(A,\mu,\partial)$. For $i=1,\dots,n$, we define a linear mapping $\D_i^\star: A \to A$ by $\D_i^\star(a) := \partial_i^\star(a \otimes 1_A)$. If each $\partial_i^\star$ is an $A$-bimodule homomorphism, then $\D^\star=(\D_1^\star,\dots,\D_n^\star)$ gives a cyclic divergence for $(A,\mu,\partial)$ compatible with $\partial^\star$. The same is true if $\D_i^\star: A \to A$ is defined instead by $\D_i^\star(a) := \partial_i^\star(1_A \otimes a)$.
\end{lemma}

\begin{proof}
Let $\iota_A: A \to A \otimes A$ be defined by $\iota_A(a) := a \otimes 1_A$ such that $\D_i^\star = \partial_i^\star \circ \iota_A$ for $i=1,\dots,n$. Using the defining property of $\partial_i^\star$ and that $\mu \circ (\partial_i^\star \otimes \id_A) = \partial_i^\star \circ (\id_A \otimes \mu)$ holds as $\partial_i^\star$ is an $A$-bimodule homomorphism, we may compute for $i,j=1,\dots,n$ that
\begin{align*}
\D_j \circ \D_i^\star &= \mu \circ \sigma \circ (\partial_j \circ \partial_i^\star) \circ \iota_A\\
                      &= \mu \circ \sigma \circ \big((\id_A \otimes \partial_i^\star) \circ (\partial_j \otimes \id_A) + \delta_{i,j} \id_{A \otimes A}\big) \circ \iota_A\\
                      &= \mu \circ (\partial_i^\star \otimes \id_A) \circ \pi_{(132)} \circ (\partial_j \otimes \id_A) \circ \iota_A + \delta_{i,j} \id_A\\
                      &= \partial_i^\star \circ (\id_A \otimes \mu) \circ \pi_{(132)} \circ (\partial_j \otimes \id_A) \circ \iota_A + \delta_{i,j} \id_A\\
											&= \partial_i^\star \circ \sigma \circ \partial_j + \delta_{i,j} \id_A,
\end{align*}
where, in the second step, we have used that $(\id_A \otimes \partial_j) \circ \iota_A = 0$, in the third step, that $\sigma \circ (\id_A \otimes \partial_i^\star) = (\partial_i^\star \otimes \id_A) \circ \pi_{(132)}$, and $$(\id_A \otimes \mu) \circ \pi_{(132)} \circ (\partial_j \otimes \id_A) \circ \iota_A = \sigma \circ \partial_j$$ in the last step. The proof for the additional statement is analogous.
\end{proof}

\subsection{Irrotational and cyclically irrotational tuples}\label{subsec:irrotational_tules}

The purpose of a divergence or a cyclic divergence, respectively, is to construct antiderivatives (modulo the associated number operators); such ``integrations'' are possible only under some suitable ``integrability condition''. In classical analysis, this is known as the Schwarz integrability condition, which corresponds in three dimensions to the vanishing of the curl; inspired by the classical terminology of irrotational (smooth) vector fields, we make the following definition.

\begin{definition}\label{def:irrotational}
Let $(A,\mu,\partial)$ be a multivariable GDQ ring with associated cyclic gradient $\D$.
\begin{enumerate}
 \item An $n$-tuple $u=(u_1,\dots,u_n)\in (A\otimes A)^n$ is said to be \emph{irrotational} if it satisfies the condition that
 $$(\id_A \otimes \partial_i)(u_j) = (\partial_j \otimes \id_A)(u_i) \qquad\text{for $i,j=1,\dots,n$}.$$
 \item An $n$-tuple $a=(a_1,\dots,a_n)\in A^n$ is said to be \emph{cyclically irrotational} if it satisfies the condition that
 $$\partial_i a_j = \sigma (\partial_j a_i) \qquad\text{for $i,j=1,\dots,n$}.$$
\end{enumerate}
\end{definition}

We collect here some important properties of the classes of irrotational and cyclically irrotational tuples. In order to simplify the notation, it is appropriate to introduce the following notation: for every weakly graded multivariable GDQ ring $(A,\mu,\partial)$ with weak grading $L$ and associated number operator $N$, we define for every $k\geq 1$ an operator $N_k: A^{\otimes k} \to A^{\otimes k}$ by
$$N_k := (k-1) \id_{A^{\otimes k}} + \sum^k_{p=1} \id_A^{\otimes (p-1)} \otimes N \otimes \id_A^{\otimes (k-p)}$$
and an operator $L_k: A^{\otimes k} \to A^{\otimes k}$ by $L_k := N_k + \id_{A^{\otimes k}}$, which can be written as
$$L_k = \sum^k_{p=1} \id_A^{\otimes (p-1)} \otimes L \otimes \id_A^{\otimes (k-p)}.$$
Note that in particular $N_1 = N$ and $N_2 = N \otimes \id_A + \id_A \otimes N + \id_{A \otimes A}$ as well as $L_1 = L$ and $L_2 = L \otimes \id_A + \id_A \otimes L$. Since $L$ is a weak grading, we have \eqref{eq:numop_coderiv-N}, from which it is possible to deduce several other commutation relations between $\partial$ and the operators $N_k$; for later use, we record that for $i=1,\dots,n$
\begin{equation}\label{eq:numop_coderiv-N2}
\begin{aligned}
(\partial_i \otimes \id_A) \circ N_2 &= N_3 \circ (\partial_i \otimes \id_A) \qquad \text{and}\\
(\id_A \otimes \partial_i) \circ N_2 &= N_3 \circ (\id_A \otimes \partial_i).
\end{aligned}
\end{equation}
Further, we notice that both $N_k$ and $L_k$ commute with $\pi_\sigma$ for each $\sigma \in S_k$.

With the following lemma, we address properties of irrotational tuples.

\begin{lemma}\label{lem:irrotational_properties}
Let $(A,\mu,\partial)$ be a multivariable GDQ ring with divergence $\partial^\star$; denote by $L$ the associated weak grading and by $N$ the corresponding number operator. Then the following statements hold true:
\begin{enumerate}
 \item If $u =(u_1,\dots,u_n) \in (A\otimes A)^n$ is irrotational, then $(N_2 u_1,\dots, N_2 u_n)$ is irrotational as well.
 \item\label{it:irrotational_properties-ii} If $N_3$ is injective and $u=(u_1,\dots,u_n) \in (A\otimes A)^n$ is such that the $n$-tuple $(N_2 u_1,\dots, N_2 u_n)$ is irrotational, then $u$ is irrotational.
 \item\label{it:irrotational_properties-iii} For each irrotational $u=(u_1,\dots,u_n) \in (A\otimes A)^n$, we have that
 $$\partial_j\big(\partial^\star u\big) = N_2 u_j \qquad\text{for $j=1,\dots,n$}$$
 and therefore
 $$(N \circ \partial^\star)(u) = \partial^\star(N_2 u_1,\dots, N_2 u_n).$$
\end{enumerate} 
\end{lemma}

\begin{proof}
(i) Since $L$ is a weak grading, we have \eqref{eq:numop_coderiv-N2}; hence, for any $n$-tuple $u=(u_1,\dots,u_n) \in (A\otimes A)^n$ which is irrotational, we obtain that
$$(\partial_i \otimes \id_A)(N_2 u_j) = N_3\big((\partial_i \otimes \id_A)(u_j)\big) = N_3\big((\id_A \otimes \partial_j)(u_i)\big) = (\id_A \otimes \partial_j)(N_2 u_i)$$
for $i,j=1,\dots,n$, which shows that $(N_2 u_1,\dots, N_2 u_n)$ is irrotational.

(ii) Let $u=(u_1,\dots,u_n) \in (A\otimes A)^n$ be such that $(N_2 u_1,\dots, N_2 u_n)$ is irrotational. We conclude from \eqref{eq:numop_coderiv-N2} that
$$N_3\big((\partial_i \otimes \id_A)(u_j)\big) = (\partial_i \otimes \id_A)(N_2 u_j) = (\id_A \otimes \partial_j)(N_2 u_i) = N_3\big((\id_A \otimes \partial_j)(u_i)\big)$$
for $i,j=1,\dots,n$. If $N_3$ is injective, then the latter implies that $u$ is irrotational.

(iii) First, we observe that we have for each irrotational $u=(u_1,\dots,u_n) \in (A\otimes A)^n$, due to the defining property \eqref{eq:GDQring_divergence} of the divergence $\partial^\star$, that
\begin{align*}
\partial_j\big(\partial^\star u\big) &= \sum^n_{i=1} \partial_j(\partial_i^\star u_i)\\
                                     &= u_j + \sum^n_{i=1} \Big((\partial_i^\star \otimes \id_A)\big((\id_A \otimes \partial_j) u_i\big) + (\id_A \otimes \partial_i^\star)\big((\partial_j \otimes \id_A) u_i\big)\Big)\\
                                     &= u_j + \sum^n_{i=1} \Big((\partial_i^\star \otimes \id_A)\big((\partial_i \otimes \id_A) u_j\big) + (\id_A \otimes \partial_i^\star)\big((\id_A \otimes \partial_i) u_j\big)\Big)\\
                                     &= N_2 u_j
\end{align*}
for $j=1,\dots,n$. From the latter, we conclude that
$$(N \circ \partial^\star)(u) = \sum^n_{j=1} \partial^\star_j\big(\partial_j(\partial^\star u)\big) = \sum^n_{j=1} \partial^\star_j\big( N_2 u_j\big) = \partial^\star(N_2 u_1,\dots, N_2 u_n).$$
This proves the assertions made in (iii).
\end{proof}

The following lemma addresses analogously the properties of cyclically irrotational tuples.

\begin{lemma}\label{lem:cyclically_irrotational_properties}
Let $(A,\mu,\partial)$ be a multivariable GDQ ring with divergence $\partial^\star$; denote by $L$ the associated weak grading and by $N$ the corresponding number operator. Furthermore, let $\D^\star$ be a cyclic divergence compatible with $\partial^\star$; denote by $C$ the corresponding cyclic symmetrization operator. Then the following statements hold true:
\begin{enumerate}
 \item If $a=(a_1,\dots,a_n)\in A^n$ is cyclically irrotational, then $(L a_1,\dots,L a_n)$ is cyclically irrotational as well.
 \item\label{it:cyclically_irrotational_properties-ii} If $L_2$ is injective and $a=(a_1,\dots,a_n) \in A^n$ is such that $(L a_1,\dots,L a_n)$ is cyclically irrotational, then $a$ is cyclically irrotational.
 \item\label{it:cyclically_irrotational_properties-iii} For each cyclically irrotational $a=(a_1,\dots,a_n) \in A^n$, we have that
 $$\D_j\big(\D^\star a\big) = L a_j \qquad\text{for $j=1,\dots,n$}$$
 and therefore
 $$(C \circ \D^\star)(a) = \D^\star(L a_1,\dots, L a_n).$$
\end{enumerate} 
\end{lemma}

\begin{proof}
(i) Since $L$ is a weak grading, we conclude with the help of \eqref{eq:numop_coderiv-L} that for every cyclically irrotational $n$-tuple $a=(a_1,\dots,a_n) \in A^n$
$$\partial_i(L a_j) = L_2(\partial_i a_j) = L_2 \big(\sigma(\partial_j a_i)\big) = \sigma\big( L_2  (\partial_j a_i)\big) = \sigma\big( \partial_j (L a_i)\big)$$
for $i,j=1,\dots,n$, which shows that $(L a_1,\dots,L a_n)$ is cyclically irrotational.

(ii) Let $a=(a_1,\dots,a_n) \in A^n$ be such that $(L a_1,\dots,L a_n)$ is cyclically irrotational. Using \eqref{eq:numop_coderiv-L}, we derive that
$$L_2\big(\sigma(\partial_j a_i)\big) = \sigma\big( L_2 (\partial_j a_i)\big) = \sigma\big( \partial_j (L a_i)\big) = \partial_i(L a_j) = L_2\big(\partial_i a_j\big)$$
for $i,j=1,\dots,n$. Now, since $L_2$ is assumed to be injective, we conclude from the latter that $a$ is cyclically irrotational.

(iii) Let $a=(a_1,\dots,a_n) \in A^n$ be cyclically irrotational. By using the defining property of the cyclic divergence $\D^\star$, we obtain that
$$\D_j \big( \D^\star a \big) = \sum^n_{i=1} \D_j (\D_i^\star a_i) = a_j + \sum^n_{i=1} \partial_i^\star\big(\sigma(\partial_j a_i)\big) =  a_j + \sum^n_{i=1} \partial_i^\star\big(\partial_i a_j) = L a_j$$
for $j=1,\dots,n$, as desired. From this, we may deduce that
$$(C \circ \D^\star)(a) = \sum^n_{j=1} \D_j^\star\big(\D_j ( \D^\star a )\big) = \sum^n_{j=1} \D_j^\star(L a_j) = \D^\star (L a_1,\dots, L a_n),$$
which is the second identity claimed in (iii). 
\end{proof}

\subsection{Universality and nc differential forms}

The theory of noncommutative differential forms and their universality property play a fundamental role in noncommutative geometry and related subjects; see \cite{K83,C85,W89,CQ95,VGB93,Gsch16}, for instance. Here, we restrict ourselves to the first order differential calculus. Our goal is to explore how multivariable GDQ rings fit into this framework.

Let us consider a unital complex algebra $A$. A tuple $(M_0,d_0)$ consisting of an $A$-bimodule $M_0$ and an $M_0$-valued derivation $d_0: A \to M_0$ is said to be \emph{universal} if it has the following universal property: every other derivation $d: A \to M$ with values in some $A$-bimodule $M$ factorizes in a unique way through $M_0$ via $d_0$, i.e., there exists a unique $A$-bimodule homomorphism $\rho: M_0 \to M$ such that $d = \rho \circ d_0$. It is obvious that the universal property characterizes a universal $(M_0,d_0)$ up to isomorphisms of $A$-bimodules; its existence is less clear but the construction is pretty standard. For that purpose, let $\mu: A \otimes A \to A$ be the multiplication map associated to $A$. The $A$-bimodule of \emph{noncommutative $1$-forms} is defined by $\Omega^1(A) := \ker \mu$; if endowed with the \emph{universal derivation} $\delta: A \to \Omega^1(A)$ by $\delta(a) := [a, 1_A \otimes 1_A]$, the tuple $(\Omega^1(A),\delta)$ satisfies the aforementioned universal property; in fact, since $\Omega^1(A)$ is the linear span of $\{a_1 \cdot \delta(a_2) \mid a_1,a_2 \in A\}$, the $A$-bimodule homomorphism $\rho: \Omega^1(A) \to M$ for a derivation $d: A \to M$ with values in some $A$-bimodule $M$ is determined by $\rho(a_1 \cdot \delta(a_2)) = a_1 \cdot d(a_2)$.

Now, let $(A,\mu,\partial)$ be a unital multivariable GDQ ring. Recall that the gradient $\partial: A \to (A \otimes A)^n$ is a derivation. Thus, by universality of $(\Omega^1(A),\delta)$, there exists a unique $A$-bimodule homomorphism $\rho_A: \Omega^1(A) \to (A \otimes A)^n$ such that $\partial = \rho_A \circ \delta$.
We say that $(A,\mu,\partial)$ is \emph{universal} if $((A\otimes A)^n,\partial)$ has the universality property introduced above, or equivalently, if $\rho_A$ is an $A$-bimodule isomorphism.
Here, we will work with some weaker constraint: we say that $(A,\mu,\partial)$ is \emph{proper} if $\rho_A$ admits at least a left inverse, i.e., if we find an $A$-bimodule homomorphism $\lambda: (A \otimes A)^n \to \Omega^1(A)$ with the property that $\lambda \circ \rho_A = \id_{\Omega^1(A)}$; in this case, we obviously have that $\delta = \lambda \circ \partial$.

\begin{lemma}\label{lem:proper}
Let $(A,\mu,\partial)$ be a unital multivariable GDQ ring. Then the following statements are equivalent:
\begin{enumerate}
 \item $(A,\mu,\partial)$ is proper.
 \item For every derivation $d: A \to M$ with values in some $A$-bimodule $M$ one finds elements $m_1,\dots,m_n \in M$ such that $$d(a) = \sum^n_{j=1} (\partial_j a) \sharp m_j \qquad\text{for all $a\in A$}.$$
 \item There are $\omega_1,\dots,\omega_n\in\Omega^1(A)$ such that $$\delta(a) = \sum^n_{j=1} (\partial_j a) \sharp \omega_j \qquad\text{for all $a\in A$}.$$
\end{enumerate}
\end{lemma}

\begin{proof}
Assume (i) and consider a derivation $d: A \to M$ with values in some $A$-bimodule $M$. By the universality of $(\Omega^1(A),\delta)$, there is a unique $A$-bimodule homomorphism $\rho: \Omega^1(A) \to M$ such that $d = \rho \circ \delta$. On the other hand, since $(A,\mu,\partial)$ is assumed to be proper, we also have that $\delta = \lambda \circ \partial$ for the left inverse $\lambda$ of $\rho_A$; thus, in summary, $d = (\rho \circ \lambda) \circ \partial$. Next, since $\rho \circ \lambda: (A \otimes A)^n \to M$ is an $A$-bimodule homomorphism, we observe that
$$(\rho \circ \lambda)(u) = \sum^n_{j=1} u_j \sharp m_j \qquad \text{for all $u=(u_1,\dots,u_n) \in (A \otimes A)^n$},$$
where $m_j := (\rho \circ \lambda)(0, \dots, 1 \otimes 1, \dots, 0)$ for the vector $(0, \dots, 1 \otimes 1, \dots, 0) \in (A \otimes A)^n$ which is zero besides the $1 \otimes 1$ in its $j$-th component. Thus, the assertion (ii) follows.

The implication ``(ii) $\Longrightarrow$ (iii)'' is trivial.

In order to prove ``(iii) $\Longrightarrow$ (i)'', we proceed as follows. Assume (iii) and define with the given $\omega_1,\dots,\omega_n\in\Omega^1(A)$ an $A$-bimodule homomorphism $\lambda: (A\otimes A)^n \to M$ by
$$\lambda(u) = \sum^n_{j=1} u_j \sharp \omega_j \qquad \text{for all $u=(u_1,\dots,u_n) \in (A \otimes A)^n$}.$$
By the choice of $\omega_1,\dots,\omega_n$, we easily see that $\lambda \circ \rho_A = \id_{\Omega^1(A)}$, which proves that $(A,\mu,\partial)$ is proper, as asserted in (i). 
\end{proof}

Remarkably, in many relevant cases, the latter result extends to a characterization of universality. 

\begin{lemma}\label{lem:universal}
Let $(A,\mu,\partial)$ be a unital multivariable GDQ ring. Suppose that there are elements $a_1,\dots,a_n\in A$ such that $\partial_j a_i = \delta_{i,j} 1_A \otimes 1_A$ for $i,j=1,\dots,n$. Then $(A,\mu,\partial)$ is universal if and only if it is proper.
\end{lemma}

\begin{proof}
It suffices to prove that $(A,\mu,\partial)$ is universal if it is proper. For that purpose, consider the $A$-bimodule homomorphism $\rho_A: \Omega^1(A) \to (A \otimes A)^n$; we will show that $\rho_A$ admits an inverse. By (the proof of) Lemma \ref{lem:proper}, we know that there are $\omega_1,\dots,\omega_n\in \Omega^1(A)$ such that $\delta(a) = \sum^n_{j=1} (\partial_j a) \sharp \omega_j$, and we have deduced that the $A$-bimodule homomorphism $\lambda: (A\otimes A)^n \to M$ defined by
$$\lambda(u) = \sum^n_{j=1} u_j \sharp \omega_j \qquad \text{for all $u=(u_1,\dots,u_n) \in (A \otimes A)^n$}$$
satisfies $\lambda \circ \rho_A = \id_{\Omega^1(A)}$; we will prove that $\rho_A \circ \lambda = \id_{(A\otimes A)^n}$. In order to do so, we first notice that $\omega_j = \delta(a_j)$ for $j=1,\dots,n$ by choice of $a_1,\dots,a_n$; hence, we get for every $u=(u_1,\dots,u_n) \in (A\otimes A)^n$ that
$$\rho_A(\lambda(u)) = \sum^n_{j=1} u_j \sharp \rho_A(\delta(a_j)) = \sum^n_{j=1} u_j \sharp (\partial a_j) = u.$$
This verifies that $\rho_A$ is an $A$-bimodule isomorphism (with inverse $\lambda$).
\end{proof}

\section{THE CHARACTERIZATION OF GRADIENTS}\label{sec:gradients}

We are now prepared for stating and proving analogues of the Theorems \ref{thm:characterization_cyclic_gradients} and \ref{thm:characterization_free_gradients} in the generality of multivariable GDQ rings.

\subsection{The cyclic gradient case}

We begin with the characterization of cyclic gradients.

\begin{theorem}\label{thm:characterization_cyclic_gradients_GDQ}
Let $(A,\mu,\partial)$ be a multivariable GDQ ring. Suppose that the following conditions are satisfied:
\begin{itemize}
 \item There exists a divergence $\partial^\star=(\partial^\star_1,\dots,\partial^\star_n)$ for $(A,\mu,\partial)$ in the sense of Definition \ref{def:GDQring_divergence} which consists of $A$-bimodule homomorphisms; denote by $L$ and $N$ the grading and the associated number operator, respectively, which are induced by $\partial^\star$ as explained in Lemma \ref{lem:divergence_induces_grading}.
 \item There exists a cyclic divergence $\D^\star=(\D_1^\star,\dots,\D_n^\star)$ compatible with $\partial^\star$ in the sense of Definition \ref{def:GDQring_cyclic_divergence}.
 \item The grading $L: A \to A$ is injective and we have for the ranges of $\D^\star: A^n \to A$ and $N: A \to A$ that $\ran \D^\star \subseteq \ran N$. 
\end{itemize}
Then, for any given $n$-tuple $a=(a_1,\dots,a_n) \in A^n$, the following statements are equivalent:
\begin{enumerate}
 \item $a$ is a cyclic gradient, i.e., there exists $b\in A$ such that $\D b = a$.
 \item $a$ is cyclically irrotational.
 \item\label{it:characterization_cyclic_gradients_GDQ_divergence} For $j=1,\dots,n$, it holds true that $\D_j \big( \D^\star a \big) = L a_j$.
\end{enumerate}
\end{theorem}

\begin{proof}
That ``(i) $\Longrightarrow$ (ii)'' holds is the content of Lemma \ref{lem:GDQring_coass_cyclic}.
The validity of the implication ``(ii) $\Longrightarrow$ (iii)'' follows from Lemma \ref{lem:cyclically_irrotational_properties} Item \eqref{it:cyclically_irrotational_properties-iii}.
Finally, ``(iii) $\Longrightarrow$ (i)'' can be shown as follows. Since $\ran \D^\star \subseteq \ran N$ by assumption, we find an element $b\in A$ such that $N b = \D^\star a$. By using the commutation relation \eqref{eq:number_operator_cyclic_derivative} provided by Lemma \ref{lem:number_operator_cyclic_derivative} and the assumption (iii), we get that
$$L (\D_j b) = \D_j (N b) = \D_j (\D^\star a) = L a_j \qquad\text{for $j=1,\dots,n$};$$
by the injectivity of $L$, we conclude that $\D_j b = a_j$ for $j=1,\dots,n$, which is (i).
\end{proof}

While in the proof of Theorem \ref{thm:characterization_cyclic_gradients_GDQ} the element $b$ with the property $\D b = a$ is found as a solution of the equation $N b = \D^\star a$, Lemma \ref{lem:symmetrization_operator_cyclic_derivative} suggests the following alternative.

\begin{lemma}\label{lem:characterization_cyclic_gradients_GDQ}
In the situation of Theorem \ref{thm:characterization_cyclic_gradients_GDQ}, suppose that $a=(a_1,\dots,a_n) \in A^n$ satisfies the equivalent conditions formulated in that theorem. Then every $b\in A$ which satisfies $C b = \D^\star a$ has the property that $\D b = a$.
\end{lemma}

\begin{proof}
Using Lemma \ref{lem:symmetrization_operator_cyclic_derivative} and Item \eqref{it:characterization_cyclic_gradients_GDQ_divergence} of Theorem \ref{thm:characterization_cyclic_gradients_GDQ}, we get that
$$L(\D_j b) = \D_j(C b) = \D_j(\D^\star a) = L a_j$$
and hence, by the injectivity of $L$, that $\D_j b = a_j$ for $j=1,\dots,n$, as asserted.  
\end{proof}

The cyclic symmetrization operator can be used to give the following two variants of Theorem \ref{thm:characterization_cyclic_gradients_GDQ}.

\begin{theorem}\label{thm:characterization_cyclic_gradients_GDQ_v2}
Let $(A,\mu,\partial)$ be a multivariable GDQ ring. Suppose that the following conditions are satisfied:
\begin{itemize}
 \item There exists a divergence $\partial^\star=(\partial^\star_1,\dots,\partial^\star_n)$ for $(A,\mu,\partial)$ in the sense of Definition \ref{def:GDQring_divergence}; denote by $L$ and $N$ the weak grading and the associated number operator, respectively, which are induced by $\partial^\star$ as explained in Lemma \ref{lem:divergence_induces_grading}.
 \item There exists a cyclic divergence $\D^\star=(\D_1^\star,\dots,\D_n^\star)$ compatible with $\partial^\star$ in the sense of Definition \ref{def:GDQring_cyclic_divergence}; denote by $C$ the associated cyclic symmetrization operator.
 \item The weak grading $L: A \to A$ is injective and we have for the ranges of $\D^\star: A^n \to A$ and $C: A \to A$ that $\ran \D^\star \subseteq \ran C$. 
\end{itemize}
Then, for any given $n$-tuple $a=(a_1,\dots,a_n) \in A^n$, the following statements are equivalent:
\begin{enumerate}
 \item $a$ is a cyclic gradient.
 \item $a$ is cyclically irrotational.
 \item For $j=1,\dots,n$, it holds true that $\D_j \big( \D^\star a \big) = L a_j$.
\end{enumerate}
\end{theorem}

\begin{proof}
The implications ``(i) $\Longrightarrow$ (ii)'' and ``(ii) $\Longrightarrow$ (iii)'' are both proven in exactly the same way as for Theorem \ref{thm:characterization_cyclic_gradients_GDQ}; the proof of the implication ``(iii) $\Longrightarrow$ (i)'' is analogous to Theorem \ref{thm:characterization_cyclic_gradients_GDQ} but follows the lines of Lemma \ref{lem:characterization_cyclic_gradients_GDQ}.
Indeed, since $\ran \D^\star \subseteq \ran C$ by assumption, we find an element $b\in A$ such that $C b = \D^\star a$. By using the commutation relation \eqref{eq:symmetrization_operator_cyclic_derivative} provided by Lemma \ref{lem:symmetrization_operator_cyclic_derivative} and the assumption (iii), we get that
$$L (\D_j b) = \D_j (C b) = \D_j (\D^\star a) = L a_j \qquad\text{for $j=1,\dots,n$};$$
by the injectivity of $L$, we conclude that $\D_j b = a_j$ for $j=1,\dots,n$, which is (i).
\end{proof}

\begin{theorem}\label{thm:characterization_cyclic_gradients_GDQ_v3}
Let $(A,\mu,\partial)$ be a multivariable GDQ ring. Suppose that the following conditions are satisfied:
\begin{itemize}
 \item There exists a divergence $\partial^\star=(\partial^\star_1,\dots,\partial^\star_n)$ for $(A,\mu,\partial)$ in the sense of Definition \ref{def:GDQring_divergence}; denote by $L$ and $N$ the weak grading and the associated number operator, respectively, which are induced by $\partial^\star$ as explained in Lemma \ref{lem:divergence_induces_grading}.
 \item There exists a cyclic divergence $\D^\star=(\D_1^\star,\dots,\D_n^\star)$ compatible with $\partial^\star$ in the sense of Definition \ref{def:GDQring_cyclic_divergence}; denote by $C$ the associated cyclic symmetrization operator.
 \item The weak grading $L$ is invertible and $L_2$ is injective. 
\end{itemize}
Then, for any given $n$-tuple $a=(a_1,\dots,a_n) \in A^n$, the following statements are equivalent:
\begin{enumerate}
 \item $a$ is a cyclic gradient.
 \item $a$ is cyclically irrotational.
 \item We have that $\D^\star a \in \ran C$ and $\D_j \big( \D^\star a \big) = L a_j$ for $j=1,\dots,n$.
\end{enumerate}
\end{theorem}

\begin{proof}
Like for the previously stated theorems, the implication ``(i) $\Longrightarrow$ (ii)'' is true tanks to Lemma \ref{lem:GDQring_coass_cyclic}.

Next, we suppose that $a=(a_1,\dots,a_n)$ is cyclically irrotational. By Lemma \ref{lem:cyclically_irrotational_properties} Item \eqref{it:cyclically_irrotational_properties-iii}, we conclude that $\D_j \big( \D^\star a \big) = L a_j$ for $j=1,\dots,n$. On the other hand, Lemma \ref{lem:cyclically_irrotational_properties} Item \eqref{it:cyclically_irrotational_properties-ii} guarantees that $a' = (a_1',\dots,a_n')$ with $a_j' := L^{-1} a_j$ for $j=1,\dots,n$ is cyclically irrotational as well; thus, we may apply Lemma \ref{lem:cyclically_irrotational_properties} Item \eqref{it:cyclically_irrotational_properties-iii} to $a'$ instead of $a$, which yields $\D^\star a = \D^\star(L a_1', \dots, La_n') = C( \D^\star a') \in \ran C$. In summary, this verifies (iii). Hence, the implication ``(ii) $\Longrightarrow$ (iii)'' is shown.

Finally, we note that the implication ``(iii) $\Longrightarrow$ (i)'' is proven precisely like in Theorem \ref{thm:characterization_cyclic_gradients_GDQ_v2}; see also the proof of Lemma \ref{lem:characterization_cyclic_gradients_GDQ}. 
\end{proof}

\subsection{The free gradient case}

Next, we address the free gradient itself.

\begin{theorem}\label{thm:characterization_free_gradients_GDQ}
Let $(A,\mu,\partial)$ be a multivariable GDQ ring. Suppose that the following conditions are satisfied:
\begin{itemize}
 \item There exists a divergence $\partial^\star=(\partial^\star_1,\dots,\partial^\star_n)$ for $(A,\mu,\partial)$ in the sense of Definition \ref{def:GDQring_divergence}; denote by $L$ and $N$ the weak grading and the associated number operator, respectively, which are induced by $\partial^\star$ as explained in Lemma \ref{lem:divergence_induces_grading}.
 \item The operator $N_2$ is injective and we have for the ranges of $\partial^\star: (A \otimes A)^n \to A$ and $N: A \to A$ that $\ran \partial^\star \subseteq \ran N$. 
\end{itemize}
Then, for any $n$-tuple $u=(u_1,\dots,u_n)\in (A \otimes A)^n$, the following statements are equivalent:
\begin{enumerate}
 \item $u$ is a free gradient, i.e., there exists $a\in A$ such that $\partial a = u$.
 \item $u$ is irrotational.
 \item For $j=1,\dots,n$, it holds true that $\partial_j \big( \partial^\star u \big) = N_2 u_j$.
\end{enumerate}
\end{theorem}

\begin{proof}
The implication ``(i) $\Longrightarrow$ (ii)'' follows from the coassociativity relation \eqref{eq:GDQring_coass} satisfied by $\partial$ according to Definition \ref{def:GDQring}. The validity of ``(ii) $\Longrightarrow$ (iii)'' is the content of Lemma \ref{lem:irrotational_properties} Item \eqref{it:irrotational_properties-iii}.
In order to prove that ``(iii) $\Longrightarrow$ (i)'', we proceed as follows. Since $\ran \partial^\star \subseteq \ran N$ by assumption, we find an element $a\in A$ such that $N a = \partial^\star u$. Since $L$ is a weak grading, the number operator $N$ satisfies \eqref{eq:numop_coderiv-N} in Definition \ref{def:gradedGDQ}; thus, we get from the assumption (iii) that
$$N_2(\partial_j a) = \partial_j(N a) = \partial_j(\partial^\star u) = N_2 u_j.$$
Since $N_2$ is assumed to be injective, we conclude from the latter that $\partial_j a = u_j$ for $j=1,\dots,n$, which proves (i).
\end{proof}

The reader might have noticed that neither the statement nor the proof of Theorem \ref{thm:characterization_free_gradients_GDQ} make use of the multiplication map $\mu$. Thus, the statement remains true in a more general setting without an underlying algebra structure; we leave the details to the reader.

The following theorem is a variant of Theorem \ref{thm:characterization_free_gradients_GDQ}; this makes use of the observation recorded in Lemma \ref{lem:irrotational_properties}.

\begin{theorem}\label{thm:characterization_free_gradients_GDQ_v2}
Let $(A,\mu,\partial)$ be a multivariable GDQ ring. Suppose that the following conditions are satisfied:
\begin{itemize}
 \item There exists a divergence $\partial^\star=(\partial^\star_1,\dots,\partial^\star_n)$ for $(A,\mu,\partial)$ in the sense of Definition \ref{def:GDQring_divergence}; denote by $L$ and $N$ the weak grading and the associated number operator, respectively, which are induced by $\partial^\star$ as explained in Lemma \ref{lem:divergence_induces_grading}.
 \item The operator $N_2$ is invertible and $N_3$ is injective. 
\end{itemize}
Then, for any $n$-tuple $u=(u_1,\dots,u_n)\in (A \otimes A)^n$, the following statements are equivalent:
\begin{enumerate}
 \item $u$ is a free gradient.
 \item $u$ is irrotational.
 \item We have that $\partial^\star u \in \ran N$ and $\partial_j( \partial^\star u ) = N_2 u_j$ for $j=1,\dots,n$.
\end{enumerate}
\end{theorem}

\begin{proof}
Like in the proof of Theorem \ref{thm:characterization_free_gradients_GDQ}, we conclude from the coassociativity relation \eqref{eq:GDQring_coass} required by Definition \ref{def:GDQring} that every free gradient in $(A \otimes A)^n$ is necessarily irrotational. 

Now, let us suppose that $u=(u_1,\dots,u_n) \in (A\otimes A)^n$ is irrotational. Like in the proof of the implication ``(ii) $\Longrightarrow$ (iii)'' for Theorem \ref{thm:characterization_free_gradients_GDQ}, we involve Lemma \ref{lem:irrotational_properties} Item \eqref{it:irrotational_properties-iii} to see that $\partial_j( \partial^\star u ) = N_2 u_j$ for $j=1,\dots,n$.
Next, we use Lemma \ref{lem:irrotational_properties} Item \eqref{it:irrotational_properties-ii} to conclude that with $u$ also $u':=(u_1',\dots,u_n')$ defined by $u_j' := N_2^{-1} u_j$ for $j=1,\dots,n$ is irrotational. Thus, we can apply Lemma \ref{lem:irrotational_properties} Item \eqref{it:irrotational_properties-iii} to $u'$ instead of $u$, which gives us that $\partial^\star u = \partial^\star(N_2 u_1', \dots, N_2 u_n') = N (\partial^\star u') \in \ran N$.
Together, this shows that $u$ satisfies (iii).

Finally, if $u$ satisfies (iii), we may proceed like in the proof of ``(iii) $\Longrightarrow$ (i)'' for Theorem \ref{thm:characterization_free_gradients_GDQ} in order to show that $u$ is a free gradient.
\end{proof}

\section{THE MULTIVARIABLE GDQ RING OF NC POLYNOMIALS}\label{sec:NCPolys}

In Section \ref{sec:gradients_NCPolys}, we will show that the major parts of Theorems \ref{thm:characterization_cyclic_gradients} and \ref{thm:characterization_free_gradients} follow from the generic results about multivariable GDQ rings which we stated in the Theorems \ref{thm:characterization_cyclic_gradients_GDQ} and \ref{thm:characterization_free_gradients_GDQ}. For that purpose, we first recall some well-known facts which set up the multivariable GDQ ring of noncommutative polynomials.

\subsection{Nc polynomials} 

In the following, $\C\langle x_1,\dots,x_n\rangle$ will denote the complex unital algebra of all \emph{noncommutative polynomials} in the formal non-commuting variables $x_1,\dots,x_n$; note that every $p\in\C\langle x_1,\dots,x_n\rangle$ can be written as a linear combination of the monomials $\{1\} \cup \{x_{i_1} x_{i_2} \cdots x_{i_k} \mid k\geq 1, 1\leq i_1,\dots,i_k \leq n\}$.

\subsection{Nc and cyclic derivatives}\label{subsec:NC_cyclic_derivatives}

On the algebra $\C\langle x_1,\dots,x_n\rangle$ of noncommutative polynomials, we define the \emph{noncommutative derivatives} as the linear maps
$$\partial_1,\dots,\partial_n:\ \C\langle x_1,\dots,x_n\rangle \to \C\langle x_1,\dots,x_n\rangle \otimes \C\langle x_1,\dots,x_n\rangle$$
which satisfy $\partial_j 1 = 0$ and whose values on all other monomials are declared to be
$$\partial_j x_{i_1} \cdots x_{i_k} = \sum^k_{p=1} \delta_{j,i_p}\, x_{i_1} \cdots x_{i_{p-1}} \otimes x_{i_{p+1}} \cdots x_{i_k}.$$
Note that, depending on the situation, we will sometimes write $\partial_{x_j}$ instead of $\partial_j$.

We turn $\C\langle x_1,\dots,x_n\rangle \otimes \C\langle x_1,\dots,x_n\rangle$ into a $\C\langle x_1,\dots,x_n\rangle$-bimodule by imposing the action that is determined by $p_1 \cdot (q_1 \otimes q_2) \cdot p_2 := (p_1 q_1) \otimes (q_2 p_2)$. It is easily seen that the noncommutative derivatives $\partial_1,\dots,\partial_n$ are $\C\langle x_1,\dots,x_n\rangle \otimes \C\langle x_1,\dots,x_n\rangle$-valued derivations on $\C\langle x_1,\dots,x_n\rangle$ which are uniquely determined by the condition that $\partial_j x_i = \delta_{i,j} 1 \otimes 1$ for $i,j=1,\dots,n$.

Further, we recall that the noncommutative derivatives are known to satisfy the \emph{(joint) coassociativity relation}
\begin{equation}\label{eq:coassociativity}
(\id \otimes \partial_i) \circ \partial_j = (\partial_j \otimes \id) \circ \partial_i \qquad\text{for $i,j=1,\dots,n$}. 
\end{equation}

Thus, if we let $\mu: \C\langle x_1,\dots,x_n\rangle \otimes \C\langle x_1,\dots,x_n\rangle \to \C\langle x_1,\dots,x_n\rangle$ be the linear map which is induced by the ordinary multiplication on $\C\langle x_1,\dots,x_n\rangle$, we may summarize these facts as follows:

\begin{proposition}\label{prop:GDQring_ncpolys}
$(\C\langle x_1,\dots,x_n\rangle,\mu,\partial)$ endowed with the \emph{free gradient $\partial=(\partial_1,\dots,\partial_n)$} is a unital multivariable GDQ ring in the sense of Definition \ref{def:GDQring}.
\end{proposition}

Finally, we remind the reader of the definition of the \emph{cyclic derivatives}
$$\D_1,\dots,\D_n:\ \C\langle x_1,\dots,x_n\rangle \to \C\langle x_1,\dots,x_n\rangle.$$
Those linear maps are defined by linear extension of $\D_j 1 = 0$ and
$$\D_j x_{i_1} \cdots x_{i_k} = \sum^k_{p=1} \delta_{j,i_p}\, x_{i_{p+1}} \cdots x_{i_k} x_{i_1} \cdots x_{i_{p-1}}.$$
These $\D_1,\dots,\D_n$ are precisely the cyclic derivatives associated to the noncommutative derivatives $\partial$ in the language of Definition \ref{def:GDQring_cyclic_derivatives} and $\D = (\D_1, \dots, \D_n)$ is the \emph{cyclic gradient}.

\subsection{The number operator}\label{subsec:number_operator}

The multivariable GDQ ring of noncommutative polynomials which we have set up in Proposition \ref{prop:GDQring_ncpolys} is also graded. To see this, let us first recall that the \emph{number operator $N$} on $\C\langle x_1,\dots,x_n\rangle$ is defined by linear extension of $N 1 := 0$ and
$$N x_{i_1} \cdots x_{i_k} := k x_{i_1} \cdots x_{i_k},\qquad k\in\N,\ 1\leq i_1,\dots,i_k \leq n.$$
It is easy to see that $N$ satisfies
\begin{equation}\label{eq:number_operator}
N p = \sum^n_{j=1} (\partial_j p) \sharp x_j \qquad\text{for all $p\in \C\langle x_1,\dots,x_n\rangle$}.
\end{equation}
The formula \eqref{eq:number_operator} suggests to write $N$ in the sense of Lemma \ref{lem:divergence_induces_grading} as the number operator associated to a suitable divergence for $(\C\langle x_1,\dots,x_n\rangle,\mu,\partial)$. Indeed, it turns out (see Remark \ref{rem:prototypical_divergence}) that $\partial^\star=(\partial_1^\star,\dots,\partial_n^\star)$ defined by
$$\partial_j^\star:\ \C\langle x_1,\dots,x_n\rangle \otimes \C\langle x_1,\dots,x_n\rangle \to \C\langle x_1,\dots,x_n\rangle,\quad u \mapsto u \sharp x_j$$
yields a divergence for $(\C\langle x_1,\dots,x_n\rangle,\mu,\partial)$ in the sense of Definition \ref{def:GDQring_divergence}; obviously, \eqref{eq:number_operator} can be rewritten as $N = \partial^\star \circ \partial$.
Put $L:=N+\id$, where $\id$ stands for the identity on $\C\langle x_1,\dots, x_n\rangle$; since each $\partial_j^\star$ is even a $\C\langle x_1,\dots,x_n\rangle$-bimodule homomorphism, we conclude from Lemma \ref{lem:divergence_induces_grading} the following.

\begin{proposition}\label{prop:GDQring_ncpolys_grading}
The operator $L$ is a grading on $(\C\langle x_1,\dots,x_n\rangle,\mu,\partial)$ in the sense of Definition \ref{def:gradedGDQ}, which is induced by the divergence $\partial^\star=(\partial_1^\star,\dots,\partial_n^\star)$.
\end{proposition}

Further, we observe that $\C\langle x_1,\dots, x_n\rangle$ decomposes into the eigenspaces of the number operator $N$ as
$$\C\langle x_1,\dots, x_n\rangle = \bigoplus_{k\geq 0} \C^{(k)}\langle x_1,\dots,x_n\rangle,$$
where the subspace $\C^{(k)}\langle x_1,\dots, x_n\rangle$ consists of all noncommutative polynomials in $\C\langle x_1,\dots, x_n\rangle$, which are homogeneous of degree $k$.
Similarly, for every integer $m\geq 1$, the algebra $\C\langle x_1,\dots, x_n\rangle^{\otimes m}$ decomposes into the eigenspaces of the operators $N_m$ and $L_m$, which were introduced in Section \ref{subsec:irrotational_tules}; indeed,
$$\C\langle x_1,\dots, x_n\rangle^{\otimes m} = \bigoplus_{k\geq 0} P^m_k,$$
where, for each $k\geq 0$, the subspace
$$P^m_k := \bigoplus_{\substack{k_1,\dots,k_m\geq 0\\ k_1 + \dots + k_m = k}} \C^{(k_1)}\langle x_1,\dots,x_n\rangle \otimes \dots \otimes \C^{(k_m)}\langle x_1,\dots,x_n\rangle$$
is the eigenspace of $N_m$ for the eigenvalue $k+m-1$ and at the same time the eigenspace of $L_m$ for the eigenvalue $k+m$. From these observations, we conclude the following.

\begin{proposition}\label{prop:GDQring_ncpolys_number-operators}
The operators $N_m$, for $m\geq 2$, and $L_m$, for $m\geq 1$, are invertible on $\C\langle x_1,\dots, x_n\rangle^{\otimes m}$. Moreover, we have that
\begin{equation}\label{eq:number_operator_surjective}
\ran N = \bigoplus_{k\geq 1} \C^{(k)}\langle x_1,\dots,x_n \rangle.
\end{equation}
\end{proposition}

\subsection{The cyclic symmetrization operator}\label{subsec:cyclic_symmetrization_operator}

We recall further that the \emph{cyclic symmetrization operator} is the linear map $C: \C\langle x_1,\dots,x_n\rangle \to \C\langle x_1,\dots,x_n\rangle$ which is determined by $C 1 := 0$ and, for all other monomials, by
$$C x_{i_1} \cdots x_{i_k} := \sum^k_{p=1} x_{i_{p+1}} \cdots x_{i_k} x_{i_1} \cdots x_{i_p}.$$
It is easily seen that for every $p\in \C\langle x_1,\dots,x_n\rangle$
$$C p = \sum^n_{j=1} x_j (\D_j p) = \sum^n_{j=1} (\D_j p) x_j.$$
Since each $\partial_i^\star$ is by definition a $\C\langle x_1,\dots,x_n\rangle$-bimodule homomorphism, we conclude the following.

\begin{proposition}\label{prop:GDQring_ncpolys_cyclic-symmetrization}
There exists a cyclic divergence $\D^\star=(\D^\star_1,\dots,\D^\star_n)$ compatible with the divergence $\partial^\star=(\partial_1^\star,\dots,\partial_n^\star)$ in the sense of Definition \ref{def:GDQring_cyclic_divergence}; it is associated to $\partial^\star$ by Lemma \ref{lem:compatible_cyclic_divergence} and hence given by $\D_j^\star p = x_j p$ for $p\in\C\langle x_1,\dots,x_n\rangle$ and $j=1,\dots,n$.
The cyclic symmetrization operator $C$ is precisely the cyclic symmetrization operator associated to $\D^\star$ in the sense of Definition \ref{def:GDQring_cyclic_divergence}.
\end{proposition}

\subsection{Universality}

Since $\C\langle x_1,\dots,x_n\rangle$ is generated as a complex unital algebra by $x_1,\dots,x_n$ and since we have $\partial_i x_j = \delta_{i,j} 1 \otimes 1$, we infer that
\begin{equation}\label{eq:nc_derivatives-universality}
d(p) = \sum^n_{j=1} (\partial_j p) \sharp d(x_j) \qquad\text{for all $p\in \C\langle x_1,\dots,x_n\rangle$}
\end{equation}
holds for every derivation $d: \C\langle x_1,\dots,x_n\rangle \to M$ with values in an arbitrary $\C\langle x_1,\dots,x_n\rangle$-bimodule $M$. Therefore, Lemmas \ref{lem:proper} and \ref{lem:universal} say that the multivariable GDQ ring $(\C\langle x_1,\dots,x_n\rangle,\mu,\partial)$ is universal. In particular, we have that
\begin{equation}\label{eq:Voi-formula}
[p, 1\otimes1] = \sum^n_{j=1} (\partial_j p) \sharp [x_j,1 \otimes 1] \qquad\text{for all $p\in \C\langle x_1,\dots,x_n\rangle$},
\end{equation}
which is a formula that was used by Voiculescu in his proof of the free Poincar\'e inequality.
Here, we show that the noncommutative derivatives are determined by this condition; the following lemma gives the precise statement.

\begin{lemma}\label{lem:coefficients_unique}
Let $p\in\C\langle x_1,\dots,x_n\rangle$ be given. Then there is a unique $n$-tuple $u=(u_1,\dots,u_n)$ of bi-polynomials in $\C\langle x_1,\dots,x_n\rangle \otimes \C\langle x_1,\dots,x_n\rangle$ such that
\begin{equation}\label{eq:coefficients_unique}
[p, 1\otimes1] = \sum^n_{j=1} u_j \sharp [x_j,1 \otimes 1]
\end{equation}
holds, and this unique $n$-tuple is given by the free gradient of $p$, i.e., $u = \partial p$.
\end{lemma}

\begin{proof}
Due to Voiculescu's formula \eqref{eq:Voi-formula}, we already know that $u = \partial p$ satisfies condition \eqref{eq:coefficients_unique}. Thus, it only remains to prove uniqueness. For doing so, suppose that \eqref{eq:coefficients_unique} holds for some $n$-tuple $u=(u_1,\dots,u_n)$ of bi-polynomials in $\C\langle x_1,\dots,x_n\rangle \otimes \C\langle x_1,\dots,x_n\rangle$. By applying $\partial_i \otimes \id$ for any $i=1,\dots,n$ to both sides of the identity \eqref{eq:coefficients_unique}, we obtain
$$(\partial_i p) \otimes 1 = u_i \sharp (1 \otimes 1 \otimes 1) + \sum^n_{j=1} (\partial_i \otimes \id)(u_j) \sharp_2 [x_j,1 \otimes 1],$$
and after applying $\id \otimes \mu$ (which has the effect that the sum on the right disappears), we conclude that $\partial_i p = u_i$. Thus, $u$ is uniquely determined by the condition \eqref{eq:coefficients_unique}, as asserted.
\end{proof}

\section{THE CHARACTERIZATION OF GRADIENTS OF NC POLYNOMIALS}\label{sec:gradients_NCPolys}

In this section, we finally prove Theorems \ref{thm:characterization_cyclic_gradients} and \ref{thm:characterization_free_gradients}. As announced before, the essential parts of those theorems are immediate consequences of the two Theorems \ref{thm:characterization_cyclic_gradients_GDQ} and \ref{thm:characterization_free_gradients_GDQ} which were formulated in the generality of multivariable GDQ rings; that the required conditions are satisfied is guaranteed by Propositions \ref{prop:GDQring_ncpolys}, \ref{prop:GDQring_ncpolys_grading}, \ref{prop:GDQring_ncpolys_number-operators}, and \ref{prop:GDQring_ncpolys_cyclic-symmetrization}.

\begin{proof}[Proof of Theorem \ref{thm:characterization_cyclic_gradients}]
First, we note that Theorem \ref{thm:characterization_cyclic_gradients_GDQ} gives the equivalence of (i), (iii), and (iv). Therefore, it suffices to prove that (i) implies (ii) and that (ii) implies (iii).
 
The validity of the implication ``(i) $\Longrightarrow$ (ii)'' can be seen as follows. Suppose that $p$ is a cyclic gradient, say $p = \D q$ for some $q\in \C\langle x_1,\dots,x_n\rangle$. Voiculescu's formula \eqref{eq:Voi-formula} applied to $q$ gives $[q, 1 \otimes 1] = \sum^n_{j=1} (\partial_i q) \sharp [x_j, 1 \otimes 1]$ and we infer by applying $\sigma$ and multiplying by $-1$ that $[q, 1 \otimes 1] = \sum^n_{j=1} [x_j,\tpartial_j q]$. Finally, applying $\mu$ to both sides of the latter identity yields $0 = \sum^n_{j=1} [x_j, \D_j q]$, which is (ii) since by assumption $p = \D q$.

In order to prove ``(ii) $\Longrightarrow$ (iii)'', we proceed as follows. Suppose that $p$ satisfies $\sum^n_{j=1} [x_j, p_j] = 0$. The latter is an identity in $\C\langle x_1,\dots,x_n\rangle$, so that $\partial_i$ for $i=1,\dots,n$ can be applied to it; this gives
$$0 = \sum^n_{j=1} \partial_i [x_j,p_j] = [1\otimes 1, p_i] + \sum^n_{j=1} [x_j, \partial_i p_j].$$
Next, we apply $\sigma$ to both sides of the latter identity, which yields
$$[p_i,1\otimes 1] = \sum^n_{j=1} (\tpartial_i p_j) \sharp [x_j, 1 \otimes 1].$$
We infer that the $n$-tuple $(\tpartial_i p_1, \dots, \tpartial_i p_n)$ satisfies condition \eqref{eq:coefficients_unique} in Lemma \ref{lem:coefficients_unique} so that $\partial_j p_i = \sigma(\partial_i p_j)$ must hold for $j=1,\dots,n$, as desired.
\end{proof}

\begin{proof}[Proof of Theorem \ref{thm:characterization_free_gradients}]
The equivalence of (i), (iii), and (iv) is established by the general Theorem \ref{thm:characterization_free_gradients_GDQ}. That further (i) and (ii) are equivalent, is the content of Lemma \ref{lem:coefficients_unique}.
\end{proof}

\section{EMBEDDINGS OF THE MULTIVARIABLE GDQ RING OF POLYNOMIALS}\label{sec:embeddings}

We have already seen that the existence of a divergence gives a rich structure to the underlying multivariable GDQ ring. The following proposition shows that the noncommutative polynomials are necessarily embedded.

\begin{proposition}\label{prop:4.5}
Let $(A,\mu,\partial)$ be a unital multivariable GDQ ring with unit $1_A$ and suppose that $\partial^\star = (\partial_1^\star, \dots, \partial_n^\star)$ is a divergence for $(A,\mu,\partial)$. Define $a=(a_1,\dots,a_n)$ by $a_i := \partial_i^\star(1_A \otimes 1_A)$ for $i=1,\dots,n$. Then the evaluation homomorphism
$$\ev_a:\ \C\langle x_1,\dots,x_n\rangle \to A,\qquad x_i \mapsto a_i$$
is injective and intertwines the derivations, i.e., the following diagram commutes for each $i=1,\dots,n$:
$$\begin{xy}
\xymatrix@R+1pc@C+3pc{
\C\langle x_1,\dots,x_n\rangle\ar[d]_{\partial_{x_i}} \ar[r]^-{\ev_a} & A\ar[d]^{\partial_i}\\
\C\langle x_1,\dots,x_n\rangle \otimes \C\langle x_1,\dots,x_n\rangle \ar[r]^-{\ev_a \otimes \ev_a} & A \otimes A
}
\end{xy}$$
\end{proposition}

\begin{proof}
It follows from the characteristic identity \eqref{eq:GDQring_divergence} that $\partial_i a_j = \delta_{i,j} 1_A \otimes 1_A$ for $i,j=1,\dots,n$. Using (a straightforward extension of) Proposition 3.17 in \cite{MSW17}, we infer from the latter that $\ev_a$ is injective and that it furthermore intertwines the derivations in the sense that $(\ev_a \otimes \ev_a) \circ \partial_{x_i} = \partial_i \circ \ev_a$ for $i=1,\dots,n$. This proves the proposition.
\end{proof}

\begin{remark}
Proposition \ref{prop:4.5} might raise the question to what extent our theory is not only an abstract way of talking just about noncommutative polynomials and their derivatives. One should, however, notice that the divergence of $A$ is not necessarily mapped to the canonical divergence of $\C\langle x_1,\dots,x_n\rangle$ under the embedding $\ev$, as $\partial^\star$ does not have to respect the multiplication on $A$. Hence we might have GDQ rings $(A,\mu,\partial)$ which are isomorphic, as linear objects and respecting the derivative, to the noncommutative polynomials, but for which we have a divergence which is not the image of the canonical one on the polynomials. The other way round, we can use this then to define on the polynomials other forms of divergences. The relevance of such possibilities might become clearer, if one has in mind that our theory can be seen as a discrete version of stochastic integration for free Brownian motion; see, for example, \cite{BSp,DS}. In this language $\partial$ corresponds to the Malliavin gradient and $\partial^\star$ is the corresponding divergence operator, which gives the integration mapping. For such stochastic integration theories, one usually has at least two prominent stochastic integrals, namely the It\^o integral and the Stratonovich integral. On a linear level both theories are isomorphic, the difference is given by their multiplicative structure and different divergence operators. In the next example we will put the discrete versions of this two integration theories in our frame.
\end{remark}

\begin{example}
To simplify the notations, we will consider first the one-dimensional case $n=1$. At the end we will also indicate how to extend this to the multivariable case, for general finite $n$.

1) On one side, we have the polynomials $\C\langle x\rangle$ in one variable $x$ (of which one might think as a semicircular variable), with the usual multiplication of polynomials and the canonical derivative
\begin{equation}\label{eq:partial-xk}
\partial x^k=\sum_{p=0}^{k-1} x^{p}\otimes x^{k-p-1}
\end{equation}
and the canonical divergence
\begin{equation}\label{eq:partial-star-xk-xl}
\partial^\star(x^k\otimes x^l)=(x^k \otimes x^l)\sharp x=x^{k+l+1}.
\end{equation}
Let us now consider the Chebyshev polynomials (of the second kind) $u_k$, which are defined recursively by
$$u_0=1,\qquad u_1=x,\qquad xu_k=u_{k+1}+u_{k-1} \quad (k\geq 1).$$
It is easy to check that those $u_k$ behave with respect to $\partial$ like the $x^k$ in \eqref{eq:partial-xk}, namely
$$\partial u_k=\sum_{p=0}^{k-1} u_{p}\otimes u_{k-p-1}.$$
The mapping $x^k\mapsto u_k$ preserves of course also the additive structure, only the multiplication is changed under this mapping to
\begin{equation}\label{eq:mult-uk-ul}
u_k u_l= u_{k+l} + u_{k+l-2} + \cdots + u_{\vert k-l\vert}.
\end{equation}
If we equip the linear span $A$ of the $u_k$ (which is $\C\langle x\rangle$) with this multiplication, then we are exactly in the setting, for $n=1$, of Proposition \ref{prop:4.5}. We could now use as a divergence on the $u_k$ the image $\tilde\partial^\star$ of the divergence on the $x^k$, i.e., 
\begin{equation}\label{eq:partial-uk-ul}
\tilde\partial^\star(u_k\otimes u_l):= u_k u_1 u_l.
\end{equation}
This corresponds to Stratonovich integration and gives nothing new compared to the polynomials. But, motivated by \eqref{eq:partial-star-xk-xl}, we can also put another divergence $\partial^\star$ on $A$, given by
\begin{equation}\label{eq:new-partial-star}
\partial^\star(u_k\otimes u_l):= u_{k+l+1}.
\end{equation}
(Note that this is indeed different from \eqref{eq:partial-uk-ul}, since the multiplication on $A$, in terms of the $u_k$, is not the same as multiplication of the $x^k$.) This divergence corresponds then to It\^o integration.

2) In the case of the $u_k$ the cyclic derivative is given by
$$\D u_k = \sum_{p=0}^{\lfloor \frac{k+1}{2} \rfloor} (k-2p) u_{k-1-2p}.$$
(Note that this is different from the action of $\D$ on $x^k$ as the cyclic gradient involves the multiplication.) The action of $\D$ takes a much simpler form on the Chebyshev polynomials of the first kind $t_k$, which are defined by the recursion
$$t_0=2,\qquad t_1=x,\qquad xt_k=t_{k+1}+t_{k-1} \quad (k\geq 1).$$
Indeed, one finds then that $\D t_{k+1} = (k+1) u_k$ for each integer $k\geq 0$. For the choice of the cyclic gradient, we have again some freedom. Since $\{u_k \mid k\geq 0\}$ forms a linear basis of $\C\langle x\rangle$, we may define a linear map $\D^\star: \C\langle x \rangle \to \C\langle x\rangle$ by $\D^\star u_k := t_{k+1}$ for every integer $k\geq 0$. We see that $\D \circ \D^\star = \partial^\star \circ \sigma \circ \partial + \id$, since both sides map $u_k$ to $(k+1) u_k$; the latter confirms that $\D^\star$ is a cyclic divergence in the sense of Definition \ref{def:GDQring_cyclic_divergence}. Further, we note that the associated cyclic symmetrization operator $C := \D^\star \circ \D$ satisfies $C t_k = k t_k$ for each $k\geq 0$; on the other hand, the number operator $N := \partial^\star \circ \partial$ satisfies $N u_k = k u_k$ for each $k\geq 0$.

3) Let us finally also address briefly the multivariable versions of the constructions from above. 
First of all, we note that the polynomials of the form
$$u_{k_1}(x_{i_1}) u_{k_2}(x_{i_2}) \cdots u_{k_d}(x_{i_d})$$
for any $d\geq 0$ (where the expression is understood as the constant polynomial $1$ in the case $d=0$) and each choice of integers $k_1,\dots,k_d\geq 1$ and $1\leq i_1,\dots,i_d\leq n$ satisfying $i_1 \neq i_2 \neq \dots \neq i_d$ constitute a linear basis of $\C\langle x_1,\dots,x_n\rangle$.
This allows us to define, as an extension of \eqref{eq:new-partial-star}, a divergence of It\^o type as follows:
\begin{align*}
\lefteqn{\partial^\star_i\big(u_{k_1}(x_{i_1}) \cdots u_{k_{r-1}}(x_{i_{r-1}}) u_{k_r}(x_{i_r}) \otimes u_{l_1}(x_{j_1}) u_{l_2}(x_{j_2}) \cdots u_{l_d}(x_{j_s})\big)}\\
&= \begin{cases}
    u_{k_1}(x_{i_1}) \cdots u_{k_{r-1}}(x_{i_{r-1}}) u_{k_r + l_1 + 1}(x_{i_r}) u_{l_2}(x_{j_2}) \cdots u_{l_d}(x_{j_s}), & \text{if $i_r = i = j_1$}\\
    u_{k_1}(x_{i_1}) \cdots u_{k_{r-1}}(x_{i_{r-1}}) u_{k_r + 1}(x_{i_r}) u_{l_1}(x_{j_1}) \cdots u_{l_d}(x_{j_s}), & \text{if $i_r = i \neq j_1$}\\
    u_{k_1}(x_{i_1}) \cdots u_{k_r}(x_{i_r}) u_{l_1+1}(x_{j_1}) u_{l_2}(x_{j_2}) \cdots u_{l_d}(x_{j_s}), & \text{if $i_r \neq i = j_1$}\\
    u_{k_1}(x_{i_1}) \cdots u_{k_r}(x_{i_r}) u_1(x_i) u_{l_1}(x_{j_1}) \cdots u_{l_d}(x_{j_s}), & \text{if $i_r \neq i \neq j_1$}
   \end{cases}
\end{align*}
We leave it to the reader to verify that $\partial^\star=(\partial_1^\star,\dots,\partial_n^\star)$ is a divergence, in the sense of Definition \ref{def:GDQring_divergence}, for the multivariable GDQ ring $(\C\langle x_1,\dots,x_n\rangle, \mu,\partial)$, which we introduced in Proposition \ref{prop:GDQring_ncpolys}.

Remarkably, also the cyclic divergence that was considered in 2) admits an extension to the multivariable case; more precisely, we define $\D_j^\star$ on any expression of the form $u_{k_1}(x_{i_1}) u_{k_2}(x_{i_2}) \cdots u_{k_d}(x_{i_d})$ for $d\geq 2$ by
\begin{align*}
\lefteqn{\D^\star_j\big(u_{k_1}(x_{i_1}) \cdots u_{k_{d-1}}(x_{i_{d-1}}) u_{k_d}(x_{i_d})\big)}\\
&= \begin{cases}
    u_{k_1+k_d+1}(x_{i_1}) \cdots u_{k_{d-1}}(x_{i_{d-1}}), & \text{if $i_d = j = i_1$}\\
    u_{k_d+1}(x_{i_d}) u_{k_1}(x_{i_1}) \cdots u_{k_{d-1}}(x_{i_{d-1}}), & \text{if $i_d = j \neq i_1$}\\
    u_{k_1+1}(x_{i_1}) \cdots u_{k_{d-1}}(x_{i_{d-1}}) u_{k_d}(x_{i_d}), & \text{if $i_d \neq j = i_1$}\\
    u_1(x_i) u_{k_1}(x_{i_1}) \cdots u_{k_d}(x_{i_d}), & \text{if $i_d \neq j \neq i_1$}
   \end{cases}
\end{align*}
and on $u_k(x_i)$ for any $k\geq 0$, which corresponds to the remaining cases $d=0$ and $d=1$, by
$$\D_j^\star u_k(x_i) = \begin{cases} t_{k+1}(x_i), & \text{if $i=j$}\\ u_1(x_j) u_k(x_i), & \text{if $i\neq j$}\end{cases}.$$
We leave it to the reader to check that $\D^\star=(\D^\star_1,\dots,\D^\star_n)$ yields, in the meaning of Definition \ref{def:GDQring_cyclic_divergence}, a cyclic divergence for $(\C\langle x_1,\dots,x_n\rangle, \mu,\partial)$ which is compatible with the divergence $\partial^\star$ that we introduced above.

It is worthwhile to check that the associated number operator $N := \partial^\star \circ \partial$ satisfies
$$N u_{k_1}(x_{i_1}) u_{k_2}(x_{i_2}) \cdots u_{k_d}(x_{i_d}) = (k_1 + \dots + k_d) u_{k_1}(x_{i_1}) u_{k_2}(x_{i_2}) \cdots u_{k_d}(x_{i_d})$$
while the cyclic symmetrization operator $C := \D^\star \circ \D$ satisfies $C t_k(x_i) = k t_k(x_i)$ and, for $d\geq 2$ and $i_d\not=i_1$,
$$C u_{k_1}(x_{i_1}) \cdots u_{k_d}(x_{i_d}) = \sum_{r=1}^d k_{r+1}u_{k_{r+1}}(x_{i_{r+1}}) \cdots u_{k_d}(x_{i_d}) u_{k_1}(x_{i_1}) \cdots u_{k_r}(x_{i_r}).$$
In the case $i_1=i_d$ one has to notice that the action of $C$ does not change if one rotates the last factor to the beginning.
\end{example}

\begin{example}
In \cite[Appendix I]{Voi10}, Voiculescu discussed the ``topological'' GDQ ring that consists of the algebra $\O(K)$ of all germs of holomorphic functions around some non-empty compact set $K\subset \C$ and the comultiplication-derivation $\partial$ that is defined by
$$(\partial f)(z_1,z_2) := \frac{f(z_1)-f(z_2)}{z_1-z_2}.$$
Further, he defined a grading $L$ on $\O(K)$ by $(Lf)(z) := z f'(z) + f(z)$; therefore, the associated number operator $N$ is $(Nf)(z) = z f'(z)$. In fact, one can easily check that $(\partial^\star F)(z) := z F(z,z)$ defines a ``topological'' divergence; obviously, we also have that $N = \partial^\star \circ \partial$.

Note that the cyclic derivative $\D$ is determined by $(\D f)(z) = f'(z)$. We find that a ``topological'' cyclic divergence compatible with $\partial^\star$ is given by $(\D^\star f)(z) := z f(z)$; thus, the associated cyclic symmetrization operator $C := \D^\star \circ \D$ coincides with $N$.
\end{example}

\section{VANISHING OF THE GRADIENTS}\label{sec:vanishing_gradients}

In \cite[Theorem 2]{Voi00a}, Voiculescu characterized those noncommutative polynomials $p\in \C\langle x_1,\dots,x_n\rangle$ that have vanishing cyclic gradient $\D p$ or which are annihilated by the cyclic symmetrization operator $C$; in fact, he obtained that
$$\ker \D = \C \oplus \{[p_1,p_2] \mid p_1,p_2 \in \C\langle x_1,\dots,x_n\rangle\} = \ker C.$$
Here, we revisit this remarkable result by using the language of multivariable GDQ rings. This is the content of the following theorem; however, it remains unclear how one recovers in that formalism that $\ker \D$ is spanned by the scalars and commutators of noncommutative polynomials.

\begin{theorem}
Let $(A,\mu,\partial)$ be a multivariable GDQ ring and let $\D$ be the associated cyclic gradient.
\begin{enumerate}
 \item It holds true that $\{[a_1,a_2] \mid a_1,a_2 \in A\} \subseteq \ker \D$; in the case that $A$ is unital with unit element $1_A$, one also has that $\C 1_A \subseteq \ker \D$. 
 \item Suppose that $\partial^\star$ is a divergence for $(A,\mu,\partial)$ and let $L$ be the induced weak grading of $(A,\mu,\partial)$ as constructed in Lemma \ref{lem:divergence_induces_grading}. Suppose further that a cyclic divergence $\D^\star = (\D_1^\star,\dots,\D_n^\star)$ compatible with $\partial^\star$ exists and let $C$ be the associated cyclic symmetrization operator. Then $\ker \D \subseteq \ker C$ holds, with equality if the grading $L$ is injective.
\end{enumerate}
\end{theorem}

\begin{proof}
(i) Every commutator $[a_1,a_2]$ with $a_1,a_2\in A$ belongs to $\ker \D$, because \eqref{eq:GDQring_Leibniz-cyclic} provides an expression for $\D_j(a_1 a_2)$ which is symmetric in $a_1$ and $a_2$; that moreover every scalar multiple of $1_A$ belongs to $\ker \D$ is trivial.

(ii) By definition of the cyclic symmetrization operator $C$, it is always true that $\ker \D \subseteq \ker C$. Under the additional assumption of injectivity of $L$, it follows from the commutation relation provided by Lemma \ref{lem:symmetrization_operator_cyclic_derivative} that also $\ker C \subseteq \ker \D$, which gives equality in this case.
\end{proof}

It is very natural to ask for an analogous description of the kernels of $\partial$. In the unital case, every scalar multiple of $1_A$ clearly lies in $\ker \partial$, but it remains to decide if the inclusion $\C 1_A \subseteq \ker\partial$ can be strict. For proper $(A,\mu,\partial)$, it is immediate from Lemma \ref{lem:proper} that $\C 1_A = \ker\partial$. This remains true in the operator-valued setting; in fact, we have the following algebraic version of \cite[Lemma 3.4]{Voi00b}.

For an arbitrary unital complex algebra $B$, we denote by $B\langle x_1,\dots,x_n\rangle$ the algebra of all \emph{$B$-valued polynomials} in the formal non-commuting variables $x_1,\dots,x_n$. On $B\langle x_1,\dots,x_n\rangle$, we define the noncommutative derivatives
$$\partial_{x_1:B},\dots,\partial_{x_n:B}:\ B\langle x_1,\dots,x_n\rangle \to B\langle x_1,\dots,x_n\rangle \otimes B\langle x_1,\dots,x_n\rangle$$
as the unique $B\langle x_1,\dots,x_n\rangle \otimes B\langle x_1,\dots,x_n\rangle$-valued derivations on $B\langle x_1,\dots,x_n\rangle$ which satisfy $\partial_{x_j:B} (x_i) = \delta_{i,j} 1 \otimes 1$ for $i,j=1,\dots,n$ and $\partial_{x_j:B} (b) = 0$ for every $b\in B$ and for $j=1,\dots,n$.

\begin{theorem}
Let $p\in B\langle x_1,\dots,x_n\rangle$ be given. Suppose that $\partial_{x_j:B} (p) = 0$ for all $j=1,\dots,n$. Then $p\in B$.
\end{theorem}

\begin{proof}
To begin with, we introduce the linear map $\ev_0: B\langle x_1,\dots,x_n\rangle \to B$ by $\ev_0(p)=p(0,\dots,0)$ for every $p\in B\langle x_1,\dots,x_n\rangle$; since $\ev_0$ clearly is a $B$-bimodule map, we get by $\id \otimes_B \ev_0$ a $B$-bimodule map on $B\langle x_1,\dots,x_n\rangle \otimes_B B\langle x_1,\dots,x_n \rangle$.
Next, we consider the derivation
$$d:\ B\langle x_1,\dots,x_n \rangle \to B\langle x_1,\dots,x_n\rangle \otimes_B B\langle x_1,\dots,x_n \rangle$$
that is given by $d(p) = p \otimes_B 1 - 1 \otimes_B p$. We observe that $d$ vanishes on $B$; thus, we get that
$$d(p) = \sum^n_{j=1} (\partial_{x_j:B} p) \sharp d(x_j).$$
So, if $\partial_{x_j:B} (p) = 0$ for $j=1,\dots,n$, we necessarily have that $d(p)=0$, and applying $\id \otimes_B \ev_0$ yields that $p = \ev_0(p) \in B$, as claimed.
\end{proof}

\begin{acknowledgements} 
We thank an anonymous referee for carefully reading our manuscript and for some valuable comments. In particular, this has led to a strengthening of the statements made in Lemma \ref{lem:irrotational_properties} Item \eqref{it:irrotational_properties-ii} and in Lemma \ref{lem:cyclically_irrotational_properties} Item \eqref{it:cyclically_irrotational_properties-ii}, which allowed to relax the assumptions of Theorem \ref{thm:characterization_cyclic_gradients_GDQ_v3} and Theorem \ref{thm:characterization_free_gradients_GDQ_v2}.
\end{acknowledgements}

\bibliographystyle{amsalpha}
\bibliography{free_differential_calculus_ref}

\end{document}